\documentclass[12pt]{article}

\usepackage{amsmath, amsthm, amssymb, latexsym,epsfig,amsthm,enumerate,multicol}
\usepackage{color}
\usepackage{graphicx}
\usepackage{nccrules}
\usepackage{xfrac}
\usepackage{hyperref}

\def \N {\mathbb N}

\def \cE {\mathcal E}
\def \cF {\mathcal F}
\def \cX {\mathcal X}

\def \cW {\mathcal W}
\def \cB {\mathcal B}
\def \cC {\mathcal C}
\def \cG {\mathcal G}
\def \cN {\mathcal N}

\def \cK {\mathcal K}

\def \cS {\mathcal S}
\def \cP {\mathcal P}
\def \cH {\mathcal H}
\def \cM {\mathcal M}
\def \cV {\mathcal V}

\def \path {\mbox{\textbf{path}}}

\def \old {{\text{\tiny old}}}

\def \bad {{{\text{\tiny bad}}}}

\def \big {{{\text{\tiny big}}}}
\def \med {{{\text{\tiny med}}}}
\def \huge {{{\text{\tiny huge}}}}

\def \lbl {{{\text{\tiny label}}}}

\newcommand{\dti}{\operatorname{DTI}}

\newcommand{\dist}{\operatorname{dist}}

\newcommand{\supp}{\operatorname{supp}}
\newcommand{\diam}{\operatorname{diam}}

\newcommand{\spn}{\operatorname{span}}

\newcommand{\Z}{\mathbb{Z}}

\newcommand{\R}{\mathbb{R}}

\newtheorem{theorem}{Theorem}[section]
\newtheorem{definition}[theorem]{Definition}
\newtheorem{proposition}[theorem]{Proposition}

\newtheorem{lemma}[theorem]{Lemma}
\newtheorem{corollary}[theorem]{Corollary}

\newtheorem{rem}[theorem]{Remark}

\newcommand{\comments}[1]{}

\makeatletter
\def\blfootnote{\gdef\@thefnmark{}\@footnotetext}
\makeatother

\begin{document}

\title{A coordinate-free proof of the finiteness principle for the Whitney extension problem}
\author{Jacob Carruth \\ email: \href{jcarruth@math.utexas.edu}{jcarruth@math.utexas.edu} \and Abraham Frei-Pearson \\ email: \href{afreipearson@math.utexas.edu}{afreipearson@math.utexas.edu} \and Arie Israel \\ email: \href{arie@math.utexas.edu}{arie@math.utexas.edu} \and Bo'az Klartag \\ email: \href{klartagb@tau.ac.il}{klartagb@tau.ac.il}} \blfootnote{
The third author is partially supported by the National Science Foundation under Grant No. 1700404.\\
The fourth author is partially supported by the European Research Foundation under Grant No. 305926.
}
\date{}
\maketitle{}

\renewcommand{\thefootnote}{\fnsymbol{footnote}} 
\renewcommand{\thefootnote}{\arabic{footnote}} 

\begin{abstract}
We present a coordinate-free version of Fefferman's solution of Whitney's extension problem in the space $C^{m-1,1}(\R^n)$. While the original argument relies on an elaborate induction on collections of partial derivatives, our proof uses the language of ideals and translation-invariant subspaces in the ring of polynomials. We emphasize the role of compactness in the proof, first in the familiar sense of topological compactness, but also in the sense of finiteness theorems arising in logic and semialgebraic geometry. 
\end{abstract}

\section{Introduction}

Whitney's extension problem asks, given a subset $E \subset \R^n$ and a function $f:E \rightarrow \R$, how can one determine whether  $f$ admits an extension $F : \R^n \rightarrow \R$ in a prescribed regularity class (e.g., H\"{o}lder, $C^m$, Sobolev, etc.)? In \cite{W1,W2,W3}, H. Whitney developed characterizations for the existence of extensions in the class $C^m$ (i.e., functions which are continuously differentiable up to order $m$). In particular, in dimension $n=1$, he proved that certain natural conditions on the continuity of the finite difference quotients of a function $f : E \rightarrow \R$ (for $E \subset \R$) are necessary and sufficient for the existence of a $C^m$-extension to the real line. In higher dimensions there is no analogue of finite difference quotients and the problem is far more difficult. Several years ago, a complete characterization of $C^m$-extendibility in arbitrary dimensions was developed by C. Fefferman \cite{F3,F5}, building on the work of Y. Brudnyi and P. Shvartsman \cite{BS1,BS2,BS3, BS4, BS5,S8, S7, S10}, who solved the extension problem in $C^{1,1}(\R^n)$, work of G. Glaeser  on $C^1$-extendibility \cite{G}, and work of E. Bierstone, P. Milman, and W. Paw{\l}ucki on $C^m$-extendibility for functions on subanalytic sets \cite{BMP1,BMP2}.

In this article we focus on the H\"{o}lder class $C^{m-1,1}(\R^n)$, consisting of all $C^{m-1}$ functions $F : \R^n \rightarrow \R$ whose $(m-1)$-st order derivatives are Lipschitz continuous. This space is equipped with a seminorm
\begin{equation}\label{norm_defn}
\| F \|_{C^{m-1,1}(\R^n)} := \sup_{x,y \in \R^n}  \left(  \sum_{|\alpha|  =  m-1 } \frac{(\partial^\alpha F(x) - \partial^\alpha F(y))^2}{|x-y|^2} \right)^{\frac{1}{2}}, \quad F \in C^{m-1,1}(\R^n),
\end{equation}
where $|\alpha| := \alpha_1 + \cdots + \alpha_n$ is the order of a multiindex $\alpha = (\alpha_1,\cdots,\alpha_n) \in \Z_{\geq 0}^n$.

In \cite{S6,S7}, Shvartsman studies Whitney's extension problem in the space $C^{1,1}(\R^n)$. One of his main results is the following  \emph{finiteness principle} (see also \cite{BS1,S8}): Suppose that the restriction of a function $f:E \rightarrow \R$ (for $E \subset \R^n$) to every subset $S \subset E$ of cardinality at most $3 \cdot 2^{n-1}$ can be extended to a function $F_S \in C^{1,1}(\R^n)$ with $\| F_S \|_{C^{1,1}(\R^n)} \leq M$. Then the function $f$ itself can be extended to a function $F \in C^{1,1}(\R^n)$ with norm $\| F \|_{C^{1,1}(\R^n)} \leq \gamma(n) M$. Brudnyi and Shvartsman conjectured in \cite{BS2, BS5} (see also \cite{S8, S9, S7}) that a similar result would hold for the entire range of H\"{o}lder spaces (i.e., for all orders of smoothness $m \geq 2$). In \cite{F1}, Fefferman verified their conjecture with the following theorem: 
\begin{theorem}[The Brudnyi-Shvartsman-Fefferman finiteness principle]
\label{main_thm} For any $m,n \geq 1$, there exist constants $C^\# \geq 1$ and $k^\# \in \N$ such that the following holds. 

Let $E \subset \R^n$ and $f : E \rightarrow \R$ be given. Suppose that there exists $M > 0$ so that for all subsets $S \subset E$ satisfying  $\#(S) \leq k^\#$ there exists a function $F^S \in C^{m-1,1}(\R^n)$ with $\| F^S \|_{C^{m-1,1}(\R^n)} \leq M$ and $F^S = f$ on $S$.

Then there exists $F \in C^{m-1,1}(\R^n)$ with $\| F \|_{C^{m-1,1}(\R^n)} \leq C^\# \cdot M$ and $F = f$ on $E$.
\end{theorem}

The finiteness principle says that a function $f: E \rightarrow \R$ admits a $C^{m-1,1}$ extension if and only if for every $k^\#$-point subset $S \subset E$, the restriction $f|_S$ admits a $C^{m-1,1}$ extension with a uniform bound on the seminorm.  The parameters $k^\#$ and $C^\#$ in Theorem \ref{main_thm} are often referred to as \emph{finiteness constants} for the function space $C^{m-1,1}(\R^n)$.

In this article we present a proof of Theorem \ref{main_thm} based on a coordinate-free version of Fefferman's stopping time argument. Our approach emphasizes the metric and symmetry structures of $\R^n$ and shortens several components of the analysis through the use of compactness arguments. Two types of  compactness are relevant here. The first is topological compactness, which is the common compactness used in Analysis. The second is logic-type compactness results from the theory of semialgebraic sets. We will explain how to replace the basis-dependent notion of \emph{monotonic multiindex sets} from Fefferman's argument with the basis-independent notion of \emph{transverse dilation-and-translation-invariant subspaces}. Our use of the latter concept is likely adaptable to the study of extension problems on sub-Riemannian manifolds, where global coordinates may be unavailable. 

Our main result is a finiteness principle for $C^{m-1,1}$-extension on finite subsets $E \subset \R^n$, where the constants depend on a parameter $\cC(E) = \cC_m(E) \in \{0,1,2,\cdots\}$, called the ``complexity'' of $E$. (See section \ref{comp_sec} for the definition of this quantity.)

\begin{theorem}\label{main_thm2}
Fix $m,n \geq 1$. There exist  constants $\lambda_1,\lambda_2 \geq 1$, determined by $m$ and $n$ such that the following holds. Fix a finite set $E \subset \R^n$ and a function $f : E \rightarrow \R$. Set $k^\# = 2^{\lambda_1 \cC(E)}$ and $C^\# = 2^{\lambda_2 \cC(E)}$. Suppose that for all subsets $S \subset E$ with $\#(S) \leq k^\#$ there exists $F^S \in C^{m-1,1}(\R^n)$ with $F^S = f$ on $S$ and $\|F^S \|_{C^{m-1,1}(\R^n)} \leq 1$.  Then there exists a function $F \in C^{m-1,1}(\R^n)$ with $F = f$ on $E$ and $\|F \|_{C^{m-1,1}(\R^n)} \leq C^\#$. 

\end{theorem}

In order to deduce Theorem \ref{main_thm} from Theorem \ref{main_thm2}, we will prove the following lemma:

\begin{lemma}\label{c_bd_lem}
There exists a constant $K_0$, determined only by $m$ and $n$, such that $\cC(E) \leq K_0$ for any finite set $E \subset \R^n$.
\end{lemma}

Together, Theorem \ref{main_thm2} and Lemma \ref{c_bd_lem} imply Theorem \ref{main_thm} in the case when $E$ is a finite subset of $\R^n$ and $M=1$. By a compactness argument involving the Arzela-Ascoli theorem, one can extend this result to infinite sets. Finally, by a trivial rescaling argument we deduce Theorem \ref{main_thm} for arbitrary $M>0$.


Fefferman's proof of Theorem \ref{main_thm} yields the constants $k^\# = \exp(\exp(\gamma D))$ and $C^\# = \exp(\exp(\gamma D))$, where $D  = \binom{n+m-1}{n}$ is the dimension of the jet space for $C^{m-1,1}(\R^n)$, or equivalently, the number of multiindices $(\alpha_1,\cdots,\alpha_n)$ of order at most $m-1$, and $\gamma > 0$ is a numerical constant independent of $m$ and $n$.  Bierstone and Milman \cite{BM1} and Shvartsman \cite{S5} independently obtain the improvement $k^\# = 2^D$ at the expense of multiplying  $C^\#$ by a multiplicative factor
which does not affect the asymptotics $C^\# = O( \exp(\exp( \gamma D)))$. In \cite{FK4}, Fefferman and Klartag show that the finiteness principle fails for $C^\# = 1 + \epsilon$ for a small absolute constant $\epsilon>0$, no matter the choice of $k^\#$.

We apply compactness arguments and algebraic methods to prove our results. For this reason, some of the constants are either inexplicit or depend poorly on $m$ and $n$. In particular, the constant $K_0$ in Lemma \ref{c_bd_lem} is not explicit. By the use of more direct methods (which will lengthen the proofs), it is possible to obtain $K_0 = \exp(\exp( \gamma D))$. This dependence is likely far from optimal. In fact, evidence suggests that it is possible to take $K_0$ to be a polynomial function of the dimension $D$. With additional work one can show that the constants $\lambda_1$ and $\lambda_2$ in Theorem \ref{main_thm2} are harmless polynomial functions of $D$. This leads us to conjecture that the finiteness principle will hold with the constants $k^\# = 2^D$ and $C^\# =  \exp(\mbox{poly}(D))$.


Throughout this paper, we will use symbols  $C,C',c,$ etc., to denote \emph{universal constants} that are determined only by $m$ and $n$. The same symbol may be used to denote a different constant in separate appearances, even within the same line. 
We are grateful to the participants of the Tenth and Eleventh Whitney Problems Workshops for their interest in our work.  We are also grateful to the National Science Foundation and the European Research Foundation for their generous financial support.


\section{Notation, definitions, and preliminary lemmas}\label{sec_not}

We write $C_{p . q}$ to denote the constant $C$ appearing in Lemma p.q, Proposition p.q, Theorem p.q, etc. See \eqref{fix_c} for a central reference of specially designated constants that arise in the last part of the paper.

Let $G \subset \R^n$ be a convex domain with nonempty interior, and let $C^{m-1,1}(G)$ be the space of real-valued functions $F: G \rightarrow \R$ whose $(m-1)$-st order partial derivatives  are Lipschitz continuous. Define a seminorm on $C^{m-1,1}(G)$ by 
\[
\| F \|_{C^{m-1,1}(G)} := \sup_{x,y \in G}  \left(  \sum_{|\alpha|  =  m-1 } \frac{(\partial^\alpha F(x) - \partial^\alpha F(y))^2}{|x-y|^2}  \right)^{\frac{1}{2}} , \;\;\; F \in C^{m-1,1}(G).
\]
The seminorm on $C^{m-1,1}(\R^n)$ is abbreviated by $\| F \| := \| F \|_{C^{m-1,1}(\R^n)}$.

Let $\cP$ be the space of polynomials of degree at most $m-1$ in $n$ real variables. Let us review some of the structure and basic properties of $\cP$. First, $\cP$ is a vector space of dimension $D := \# \{ \alpha \in \Z_{\geq 0}^n : |\alpha| \leq m-1 \}$. For $x \in \R^n$, define an inner product on $\cP$:
\[
\langle P,Q \rangle_x := \sum_{|\alpha| \leq m-1} \frac{1}{\alpha !} \cdot \partial^\alpha P(x) \cdot \partial^\alpha Q(x),
\]
where $\alpha ! = \prod_{i=1}^n \alpha_i !$ and we also set $x^\alpha = \prod_{i=1}^n x_i^{\alpha_i}$. If $P(z) = \sum_{|\alpha| \leq m-1} a_\alpha \cdot (z-x)^\alpha$ and $Q(z) = \sum_{|\alpha| \leq m-1} b_\alpha \cdot (z-x)^\alpha$, then $\langle P,Q \rangle_x = \sum_{|\alpha| \leq m-1} \alpha ! \cdot  a_\alpha b_\alpha$. Therefore, the inner product space $(\cP, \langle \cdot, \cdot \rangle_x)$ admits an orthonormal basis of monomials $\{ \frac{1}{\sqrt{\alpha !} }\cdot (z - x)^\alpha \}_{|\alpha| \leq m-1}$.  We define a norm on $\cP$ by $| P |_x := \sqrt{ \langle P,P \rangle_x}$.

We define translation operators $T_{h} : \cP \rightarrow \cP$ (for $h \in \R^n$) by $T_{h}(P)(z) := P(z-h)$, and dilation operators $\tau_{x,\delta} : \cP \rightarrow \cP$ (for $(x,\delta) \in \R^n \times (0,\infty)$) by $\tau_{x,\delta}(P)(z) := \delta^{-m} P(x + \delta \cdot (z -x))$. The dilation operators lead us to define a scaled inner product on $\cP$: For $(x,\delta) \in \R^n \times (0,\infty)$, let
\[
\langle P, Q \rangle_{x, \delta} := \langle \tau_{x, \delta} (P), \tau_{x, \delta}
(Q) \rangle_x \quad \quad \quad (P, Q \in \cP),
\]
and the corresponding scaled norm is denoted by $|P|_{x, \delta} := \sqrt{ \langle P, P \rangle_{x, \delta}}$. The unit ball associated to this norm is the subset 
\[ 
\cB_{x, \delta} := \biggl\{ P :  \; |P|_{x,\delta} = \biggl( \sum_{|\alpha| \leq m-1} \frac{1}{\alpha !} \cdot ( \delta^{|\alpha| - m} \cdot \partial^{\alpha} P(x) )^2 \biggr)^{\frac{1}{2}} \leq 1 \biggr\} \subset  \cP.
\]
We write $\langle \cdot, \cdot \rangle$ and $| \cdot |$ to denote the ``standard'' inner product $\langle \cdot, \cdot \rangle_{0,1}$ and norm $| \cdot |_{0,1}$ on $\cP$, and $\cB = \cB_{0,1}$ for the corresponding unit ball. 

Given $\Omega \subset \cP$, $P_0 \in \cP$, and $r \in \R$, let $r \Omega := \{ r P : P \in \Omega\}$ and $P_0 + \Omega := \{ P_0 + P : P \in \Omega\}$. For future use, we record below a few identities and inequalities which connect the dilation and translation operators with the scaled inner products, norms, and balls.

\begin{multicols}{2}
  \begin{enumerate}[(a)]
  \item \begin{enumerate}[(i)]
\item $T_{h_1}  \circ T_{h_2} = T_{h_1+h_2}$.
\item $\tau_{x,\delta_1} \circ \tau_{x,\delta_2} = \tau_{x, \delta_1\delta_2}$.
\item $T_h \circ \tau_{x,\delta} = \tau_{x+h,\delta} \circ T_h$.
\end{enumerate}
  \item \begin{enumerate}[(i)]
\item $\langle \tau_{x,\rho} (P), \tau_{x,\rho}(Q) \rangle_{x,\delta} = \langle P, Q \rangle_{x, \delta \rho}$.
\item $| \tau_{x,\rho} (P) |_{x,\delta} = |P|_{x, \delta\rho}$.
\item $\tau_{x,\rho} \cB_{x,\delta} = \cB_{x, \delta/\rho}$.
\end{enumerate}
  \item \begin{enumerate}[(i)]
\item $\langle T_h(P) , T_h(Q) \rangle_{x,\delta} = \langle P,Q \rangle_{x - h,\delta}$.
\item $|T_h(P) |_{x,\delta} = | P |_{x-h,\delta}$.
\item $T_h \cB_{x,\delta} = \cB_{x+h,\delta}$.
\end{enumerate}
\end{enumerate}
\end{multicols}

Furthermore, for any $\delta \geq \rho > 0$,
\begin{equation}
\label{eq:change_delta}
\left\{
\begin{aligned}
&( \rho/\delta )^{m} \cdot |P|_{x,\rho} \leq |P|_{x, \delta} \leq (\rho/\delta) \cdot |P|_{x, \rho}, \; \mbox{and hence} \\
&(\delta/\rho) \cdot \cB_{x,\rho} \subset \cB_{x,\delta} \subset (\delta/\rho)^m \cdot \cB_{x,\rho}.
\end{aligned}
\right.
\end{equation}

Let $J_x F \in \cP$ denote the $(m-1)$-jet of a function $F \in C^{m-1,1}(\R^n)$ at $x$, namely, the Taylor polynomial
\[
(J_x F)(z) := \sum_{|\alpha| \leq m-1} \frac{1}{\alpha!} \cdot \partial^\alpha F(x) \cdot (z-x)^\alpha \quad (z \in \R^n).
\] 
The importance of the norms $| \cdot |_{x,\delta}$ on $\cP$ stems from the Taylor and Whitney theorems. According to Taylor's theorem, if $F \in C^{m-1,1}(G)$, where $G$ is any convex domain in $\R^n$ with nonempty interior, then
\[
| \partial^\beta ( F - J_y F )(x)| \leq C \cdot \| F \|_{C^{m-1,1}(G)} \cdot |x-y|^{m-|\beta|}, \quad \mbox{for } x,y \in G, \; |\beta| \leq m-1.
\]
This implies
\begin{equation}
\label{taylor_thm}
\left\{
\begin{aligned}
& |J_x F - J_y F|_{x, \delta} \leq C_T \|F\|_{C^{m-1,1}(G)} , \mbox{ or equivalently}  \\
& J_x F - J_y F \in C_T \|F\|_{C^{m-1,1}(G)}  \cdot  \cB_{x,\delta} \quad \mbox{for }  x,y \in G, \; \delta \geq |x-y|,
\end{aligned}
\right.
\end{equation}
where $C_T = C_T(m,n)$ is a constant determined by $m$ and $n$. Therefore the norm $| \cdot |_{x,\delta}$ may be used to describe the compatibility conditions on the $(m-1)$-jets of a $C^{m-1,1}$ function at two points $x,y$ in $\R^n$, whenever $|x-y| \leq \delta$. The conditions in \eqref{taylor_thm} capture the essence of the concept of a $C^{m-1,1}$ function in the following sense: Whitney's theorem \cite{W1} states that whenever $E \subset \R^n$ is an arbitrary set, $M > 0$, and $\{P_x\}_{x \in E}$ is a collection of polynomials with
\begin{equation}
\label{whit_cond}
|P_x - P_y|_{x,\delta} \leq M \qquad \mbox{for } x,y \in E, \; \delta = |x-y|,
\end{equation}
then there exists a $C^{m-1,1}$ function $F : \R^n \rightarrow \R$ with $\| F \| \leq CM$ and $J_x F = P_x$ for all $x \in E$. As usual, $C$ is a constant depending solely on $m$ and $n$.

The vector space of $(m-1)$-jets is a ring, denoted by $\cP_x$, equipped with the product $\odot_x$ (indexed by a basepoint $x \in \R^n$) defined by $P \odot_x Q = J_x( P \cdot Q)$. The product and translation/dilation operators are related by
\begin{equation}
\left\{
\begin{aligned}
&\tau_{x, \delta} \left( P \odot_x Q \right) = \delta^{m} \cdot \tau_{x, \delta}(P) \odot_x \tau_{x, \delta}(Q), \\
&T_h\left( P \odot_x Q \right)  = T_h(P) \odot_{x+h} T_h(Q) \qquad\qquad  \mbox{for } x, h \in \R^n, \; \delta>0.
\end{aligned} \label{eq_2115}
\right.
\end{equation}
The following lemma, taken verbatim from \cite[section 12]{FK2}, summarizes a few basic properties of the product and norms  introduced above. See the proof of Lemma 1 in \cite[section 12]{FK2} for a direct argument that leads to explicit constants. Our argument below emphasizes the r\^ole of rescaling and compactness.

\begin{lemma}
Let $x,y \in \R^n$ and $\delta, \rho > 0$. Assume that $|x-y| \leq \rho \leq \delta$. Then for any $P, Q \in \cP$,
\begin{enumerate}[(i)]
\item $\displaystyle |P|_{y, \rho} \leq C |P|_{x, \rho}$.
\item $\displaystyle \left|P \odot_x Q \right|_{x, \rho} \leq C \delta^{m} |P|_{x, \delta} |Q|_{x,
\rho}$.
\item $\displaystyle  \left|(P \odot_y Q) - (P \odot_x Q) \right|_{x, \rho}  \leq C \delta^{m} |P|_{x, \delta} |Q|_{x,
\delta}$.
\end{enumerate} \label{lem_c1}
Here, $C > 0$ is a constant depending solely on $m$ and $n$.
\end{lemma}
\begin{proof}
The main step is to use \eqref{eq_2115} and observe that by translating and rescaling, we may reduce matters to the case $x =0$ and $\rho = 1$.  Next, note that it suffices to prove the lemma for non-zero polynomials $P$ and $Q$. Normalizing, we assume that $|P|_{0,1} = |Q|_{0,1} = 1$.

To prove (i), observe that the space of all relevant parameters is compact, since $|y| \leq 1$ and $|P|_{0,1} = 1$. The left-hand side of (i) is a continuous function on this space of parameters, hence the maximum is attained, and yields the constant $C$ on the right-hand side. 

To prove (ii), observe that the left-hand side in (ii) is bounded from above by a constant $C$ by compactness, while
\[
\delta^{m} |P|_{0, \delta} \geq |P|_{0, 1} = 1
\]
for any $\delta \geq 1$, according to \eqref{eq:change_delta}. Hence (ii) holds true as well.

To prove (iii), it is more convenient to rescale so that $\delta = 1$, rather than $\rho = 1$. We may still assume that $|P|_{0,1} = |Q|_{0,1} = 1$. Consider the unit ball $B = \{ x \in \R^n : |x| \leq 1 \}$ and the function $F(x) = P(x) Q(x)$. Yet another compactness argument yields that $\| F \|_{C^{m-1,1}(B)} \leq C_0$ for a constant $C_0$ determined by $m$ and $n$. From Taylor's theorem, rendered above as \eqref{taylor_thm},
\[
\left| (P \odot_y Q) - (P \odot_0 Q) \right|_{0, \rho} = \left| J_y F - J_0 F \right|_{0, \rho} \leq C_T \cdot C_0,
\]
and the lemma is proven.
\end{proof}

Suppose $|x-y| \leq \lambda \delta$ for $\lambda \geq 1$. By \eqref{eq:change_delta} we have $|P|_{y,\delta} \leq \lambda^m |P|_{y,\lambda \delta}$ and $ |P|_{x, \lambda \delta} \leq \lambda^{-1} |P|_{x, \delta}$. Furthermore, by case (i) of Lemma \ref{lem_c1} we have $|P|_{y,\lambda \delta} \leq C |P|_{x, \lambda \delta}$. Combining these estimates gives  the inequality
\begin{equation}
\label{trans_norm}
 |P|_{y,\delta} \leq C \lambda^{m-1} |P|_{x,\delta} \qquad (x,y \in \R^n, \; |x-y| \leq \lambda \delta, \; \lambda \geq 1, \; \delta > 0).
\end{equation}
We note that \eqref{trans_norm} is equivalent  to the inclusion $\cB_{x,\delta} \subset C  \lambda^{m-1} \cB_{y,\delta}$. 

Suppose $\theta \in  C^{m-1,1}(\R^n)$  is supported on a ball $B \subset \R^n$.  Then we claim that
\begin{equation} 
\label{cutoff_jet}
| J_x(\theta) |_{x,\diam(B)} \leq C_T \| \theta \| \qquad (x \in \R^n).
\end{equation}
The inequality \eqref{cutoff_jet} is trivial if $x \in \R^n \setminus B$, as then $J_x(\theta) = 0$. Fix $x_0 \in \partial B$. Then $J_{x_0}(\theta) = 0$. Using that $|x-x_0| \leq \diam(B)$ for $x \in B$, we apply Taylor's theorem (rendered as \eqref{taylor_thm}) and obtain $ | J_x(\theta) |_{x,\diam(B)} =  | J_x(\theta)  -  J_{x_0}(\theta) |_{x,\diam(B)} \leq C_T \| \theta \|$, which yields \eqref{cutoff_jet}.

We next give a more general form of Lemma \ref{lem_c1}(iii) involving products of up to three polynomials which are allowed to vary from point to point.
\begin{lemma}\label{lem_prod_cont}
Fix polynomials $P_x,Q_x,R_x$ and $P_y,Q_y,R_y$ in $\cP$, for $|x-y| \leq \rho \leq \delta$. Suppose that $P_x,P_y \in M_0 \cB_{x,\delta}$, $Q_x,Q_y \in M_1 \cB_{x,\delta}$, and $R_x,R_y \in M_2\cB_{x,\delta}$. Also suppose that $P_x - P_y \in M_0 \cB_{x,\rho}$, $Q_x - Q_y \in M_1 \cB_{x,\rho}$, and $R_x - R_y \in M_2  \cB_{x,\rho}$. Then
\[
| P_x \odot_x Q_x \odot_x R_x - P_y \odot_y Q_y \odot_y R_y |_{x,\rho} \leq C \delta^{2m} M_0 M_1 M_2,
\]
where $C$ is a constant determined by $m$ and $n$.
\end{lemma}
\begin{proof}
In view of \eqref{eq_2115}, we may assume that $\delta = 1$. By renormalizing, we may assume $M_0=M_1=M_2=1$. Then all six polynomials belong to $\cB_{x,1}$, and the three differences $P_x - P_y$, $Q_x-Q_y$, and $R_x-R_y$ belong to $\cB_{x,\rho}$. The letter $x$ appears five times in the expression $P_x \odot_x Q_x \odot_x R_x$, and we will change these five $x$'s to five $y$'s one by one. We first apply Lemma \ref{lem_c1}(ii) three times and replace $R_x$, $Q_x$, and $P_x$ by $R_y$, $Q_y$, and $P_y$, in that respective order, as follows:
\[
| P_x \odot_x Q_x \odot_x R_x - P_y \odot_x Q_y \odot_x R_y |_{x,\rho} \leq C.
\]
This step also requires the bounds $|P_x \odot_x Q_x |_{x,1} \leq C$, $|P_x \odot_x R_y |_{x,1} \leq C$, and $|Q_y \odot_x R_y |_{x,1} \leq C$, which are all consequences of Lemma \ref{lem_c1}(ii). Next we apply Lemma \ref{lem_c1}(iii) twice, and deduce that
\[
|P_y \odot_x Q_y \odot_x R_y - P_y \odot_y Q_y \odot_y R_y|_{x,\rho} \leq C.
\]
This step requires the bounds $|P_y \odot_x Q_y|_{x,1} \leq C$ and $| Q_y \odot_y R_y|_{x,1} \leq C$, which follow from Lemma \ref{lem_c1}(ii) and, for the second inequality, also Lemma \ref{lem_c1}(iii). This concludes the proof of the lemma.
\end{proof}
\begin{rem}\label{rem1}
We can obtain a version of Lemma \ref{lem_prod_cont} also for products of two polynomials. Notice that $1 \in \delta^{-m} \cB_{x,\delta}$ for any $\delta > 0$. Thus, by taking $P_x=P_y=1$, under the hypotheses of Lemma \ref{lem_prod_cont}, $ | Q_x \odot_x R_x - Q_y \odot_y R_y |_{x,\rho} \leq C \delta^{m} M_1 M_2$.
\end{rem}

Finally, we state a few elementary facts from convex geometry. A convex set $\Omega$ in a finite-dimensional vector space $\cV$ is said to be \emph{symmetric} if $P \in \Omega \implies - P \in \Omega$. If $A$, $K$, and $T$ are symmetric convex sets then
\begin{equation}
\label{fact1}
K \subset T \implies (A+K)\cap T \subset (A \cap 2T) + K,
\end{equation}
and also if $K$ is bounded then
\begin{equation}
\label{fact2}
K \subset T + K/3 \implies K \subset 2T.
\end{equation}
To prove \eqref{fact1}, pick $x \in (A + K) \cap T$. Then $x = a + k$ with  $a \in A$ and $k \in K$. It suffices to show that $a \in 2T$. This holds since $a = x - k \in T - K \subset 2T$. Next observe that the condition $K \subset T + K/3$ implies $\sup_{x \in K} f(x) \leq \sup_{x \in T} f(x) + \frac{1}{3} \sup_{x \in K} f(x)$ for any linear functional $f: \cV \rightarrow \R$. If $K$ is bounded, this implies $\frac{2}{3} \sup_{x \in K} f(x) \leq \sup_{x \in T} f(x)$. From the Hahn-Banach theorem, $K$ is contained in the closure of $\frac{3}{2} T$, and therefore $K \subset 2 T$.

\subsection{Taylor polynomials of functions with prescribed values.}

Fix a finite subset $E \subset \R^n$ and a function $f : E \rightarrow \R$ satisfying the hypothesis of Theorem \ref{main_thm2}. That is, we assume that for some natural number $k^\# \in \N$, the following holds:
\begin{equation}\label{fin_hyp}
\cF\cH(k^\#) \left\{
\begin{aligned}
  &\quad \mbox{For all } S \subset E \mbox{ with } \#(S) \leq k^\#  \mbox{ there exists } F^S \in C^{m-1,1}(\R^n)  \\
& \qquad \mbox{ with }  F^S = f \mbox{ on } S \mbox{ and } \|F^S \| \leq 1.
\end{aligned}
\right.
\end{equation}
We call $\cF\cH(k^\#)$ the \emph{finiteness hypothesis} and $k^\#$ the \emph{finiteness constant}. We aim to construct a function $F \in C^{m-1,1}(\R^n)$ satisfying $F = f$ on $E$ and $\| F \| \leq C^\#$ for a suitable constant $C^\# \geq 1$. We first introduce a family of convex subsets of $\cP$ that contain information on the Taylor polynomials of extensions associated to subsets of $E$:
\begin{align*}
\Gamma_S(x,f,M) := \{ J_x F : & F \in C^{m-1,1}(\R^n), \;  F=f \; \mbox{on} \; S, \;  \| F \| \leq M \}, \\
&  \mbox{for} \; S \subset E, \; x \in \R^n, \; f : E \rightarrow \R, \; \mbox{and } M > 0.
\end{align*} 
We also denote $\Gamma(x,f,M) := \Gamma_E(x,f,M)$. Notice that $\Gamma_S(x,f,M)$ is nonempty if and only if there exists an extension of the restricted function $f|_S$ with $C^{m-1,1}$ seminorm at most $M$. Therefore the finiteness hypothesis $\cF\cH(k^\#)$ is equivalent to the condition that $\Gamma_S(x,f,1) \neq \emptyset$ for all $S \subset E$ with $\#(S) \leq k^\#$. Now, for $\ell \in \Z_{\geq 0}$ we define
\[
\begin{aligned}
\Gamma_\ell(x,f,M) := \{ P \in \cP : \forall S \subset E, \;  \#(S) \leq & (D +1)^\ell, \; \exists F^S \in C^{m-1,1}(\R^n), \; \\
& F^S = f \; \mbox{on} \; S, \; J_x F^S = P, \; \| F^S \| \leq M \};
\end{aligned}
\]
here, recall that $D = \dim \cP$. In other words, an element of $\Gamma_\ell(x,f,M)$ is simultaneously the jet of a solution to any extension problem associated to a subset $S \subset E$ of cardinality at most $(D+1)^\ell$. The sets denoted by $\Gamma_\ell(\cdot,\cdot,\cdot)$ were introduced in  \cite{F1} as a tool to demonstrate that $\Gamma(x,f,M)$ is nonempty -- the latter condition is relevant because it implies, in particular, the existence of an extension of $f$ with $C^{m-1,1}$ seminorm at most $M$. We note the identity
\begin{equation}\label{gamma_l}
\Gamma_\ell(x,f,M) = \bigcap_{ S \subset E, \; \#(S) \leq (D+1)^\ell} \Gamma_S(x,f,M).
\end{equation}
Given $x \in \R^n$ and $S \subset E$, let 
\[
\sigma(x,S) := \{ J_x \varphi :  \varphi \in C^{m-1,1}(\R^n), \; \varphi = 0 \; \mbox{on} \; S, \; \| \varphi \| \leq 1 \},
\] 
and given $\ell \in \Z_{\geq 0}$, let 
\begin{equation}\label{sigma_l}
\sigma_\ell(x) = \bigcap_{ S \subset E, \; \#(S) \leq (D+1)^\ell} \sigma(x,S).
\end{equation}
We also denote $\sigma(x) := \sigma(x,E)$. 

Note that $\sigma(x)$ and $\sigma_\ell(x)$ are symmetric convex sets in $\cP$, whereas $\Gamma(x,f,M)$ and $\Gamma_\ell(x,f,M)$ are merely convex. By a straightforward application of the Arzela-Ascoli theorem one can show that $\sigma(x)$, $\sigma_\ell(x)$, $\Gamma(x,f,M)$, and $\Gamma_\ell(x,f,M)$ are closed. Finally, we observe that $\sigma(x,S) = \Gamma_S(x,0,1)$, $\sigma_\ell(x) = \Gamma_\ell(x,0,1)$, and $\sigma(x) = \Gamma(x,0,1)$. 

\begin{lemma}[Relationship between $\Gamma_\ell$ and $\sigma_\ell$] \label{gamma-sigma_lem} For any $\ell \in \Z_{\geq 0}$,
\[
\begin{aligned}
&\Gamma_\ell(x,f,M/2) + (M/2) \sigma_\ell(x) \subset \Gamma_\ell(x,f,M), \;\; \mbox{and}\\
&\Gamma_\ell(x,f,M) - \Gamma_\ell(x,f,M) \subset 2M \sigma_\ell(x).
\end{aligned}
\]
\end{lemma}
\begin{proof}
By definition we have $\Gamma_S(x,f,M/2) + (M/2) \sigma(x,S) \subset \Gamma_S(x,f,M)$ and $\Gamma_S(x,f,M) - \Gamma_S(x,f,M) \subset 2M \sigma(x,S)$. The conclusion of the lemma then follows from the definition of $\Gamma_\ell$ and $\sigma_\ell$ in \eqref{gamma_l} and \eqref{sigma_l}.
\end{proof}
\begin{rem}
Lemma \ref{gamma-sigma_lem} implies that $P_x + \frac{M}{2} \cdot \sigma_\ell(x) \subset \Gamma_\ell(x,f,M) \subset P_x + 2M \cdot \sigma_\ell(x)$, for any $P_x \in \Gamma_\ell(x,f,M/2)$. Later on we will be concerned with the geometry of the set $\Gamma_\ell(x,f,M)$ at various points $x \in \R^n$. Lemma \ref{gamma-sigma_lem} implies that it is sufficient to understand the geometry of the set $\sigma_\ell(x)$ (which depends on fewer parameters and is therefore more manageable).
\end{rem}

Recall the translation and scaling transformations $T_h$ and  $\tau_{x,\delta}$ on $\cP$. With a slight abuse of notation, we also denote the transformations $T_h$ and $\tau_{x,\delta}$ on $\R^n$ given by
\[
T_h(y) = y+h, \;\;\; \tau_{x,\delta}(y) = x + \delta \cdot (y-x) \qquad (x,y,h \in \R^n, \delta > 0).
\]
Then,
\begin{equation}\label{eq_1503}
\sigma( T_h(y), T_h(S)) = T_h \left\{ \sigma(y,S) \right\}, \mbox{ and } \sigma(\tau_{x,\delta}(y), \tau_{x,\delta}(S)) = \tau_{x,\delta} \left\{ \sigma(y,S) \right\},
\end{equation}
for any $x,y, h \in \R^n$, $\delta > 0$, and $S \subset \R^n$, as may be verified directly. Here in our notation, if $T : \R^n \rightarrow \R^n$ then $T(S) = \{ T(y) : y \in S \}$. 

In the next lemma we establish two important properties of the sets $\Gamma_\ell(x,f,M)$. We show that the finiteness hypothesis $\cF\cH(k^\#)$ (see \eqref{fin_hyp}) implies that $\Gamma_\ell(x,f,M)$ is non-empty if $\ell$ and $k^\#$ are suitably related and if $M \geq 1$. We also show that the mappings $x \mapsto \Gamma_\ell(x,f,M)$ are ``quasicontinuous'' in a sense to be made precise below.
\begin{lemma}\label{gamma_trans_lem}
If $x \in \R^n$, $(D+1)^{\ell+1} \leq k^\#$, and $M \geq 1$, then
\begin{equation}
\label{FH}
\cF\cH(k^\#) \implies \Gamma_\ell(x,f,M) \neq \emptyset.
\end{equation}
If $x,y \in \R^n$, $\ell \geq 1$, $\delta \geq |x-y|$, and $M>0$, then
\begin{equation}
\label{eq_1113}
\Gamma_\ell(x,f,M) \subset \Gamma_{\ell-1}(x,f,M) + C_T M\cdot \cB_{x,\delta}
\end{equation}
and
\begin{equation}
\label{eq_1114}
\sigma_\ell(x) \subset \sigma_{\ell-1}(x) + C_T \cdot \cB_{x,\delta},
\end{equation}
where $C_T$ is the constant in \eqref{taylor_thm}.
\end{lemma}
\begin{proof}
We first show that the finiteness hypothesis with constant $k^\# \geq (D+1)^{\ell+1}$ implies the intersection of the sets in \eqref{gamma_l} is nonempty for $M=1$. As $\Gamma(x,f,M) \supset \Gamma(x,f,1)$ for $M \geq 1$, the implication \eqref{FH} will then follow. By Helly's theorem and the fact that $\dim \cP = D$, it suffices to show that the intersection of any $(D+1)$-element subcollection is nonempty. Fix $S_1,\cdots,S_{D+1} \subset E$ with $\#(S_i) \leq (D+1)^\ell$. Let  $S := S_1\cup \cdots \cup S_{D+1}$. Note that  $\Gamma_{S_1}(x,f,1) \cap \cdots \cap \Gamma_{S_{D+1}}(x,f,1) \supset \Gamma_{S}(x,f,1)$. Furthermore, $\#(S) \leq (D+1) \cdot (D+1)^{\ell} \leq k^\#$, and so $\Gamma_{S}(x,f,1) \neq \emptyset$ by the finiteness hypothesis $\cF\cH(k^\#)$. This finishes the proof of \eqref{FH}.

To prove \eqref{eq_1113} and \eqref{eq_1114} we reproduce the proof of \cite[Lemma 10.2]{F1}. Note \eqref{eq_1114} is a special case of (\ref{eq_1113}), as $\sigma_\ell(x) = \Gamma_\ell(x, 0, 1)$. So it suffices to prove (\ref{eq_1113}). Given $P \in \Gamma_\ell(x,f, M)$, we will find $Q \in \Gamma_{\ell-1}(y, f,M)$ with
\begin{equation} 
|P - Q|_{x, \delta} \leq C_T M. 
\label{eq_1208} 
\end{equation} 
For a subset $S \subset E$, consider 
\[
\cK(S) := \left \{ J_y F :  F \in C^{m-1,1}(\R^n), \; F = f \mbox{ on } S, \ \| F
\| \leq M, \ J_x F = P \right \}.
\]
Then $\cK(S) \subset \cP$ is convex, and according to \eqref{taylor_thm},
\begin{equation}\label{eq_1209}
\cK(S) \subset P + C_T M \cdot \cB_{x, \delta}.
\end{equation}
Note that $\cK(S) \neq \emptyset$ whenever $\#(S) \leq (D+1)^{\ell}$, due to the fact that $P \in \Gamma_\ell(x, f, M)$. We will show that
\begin{equation}  
\emptyset \neq \bigcap_{\substack{S \subset E \\ \#(S) \leq (D+1)^{\ell-1}}} \cK(S)
\subset \Gamma_{\ell-1}(y, f, M). \label{eq_1138}
\end{equation}
The inclusion on the right-hand side of \eqref{eq_1138} is immediate from the definition of $\Gamma_{\ell-1}(y, f, M)$. All that remains is to show that the intersection of the collection of sets in (\ref{eq_1138}) is non-empty. By Helly's theorem it suffices to show that the intersection of any $(D+1)$-element subcollection is nonempty. Thus, pick $S_1,\ldots,S_{D+1} \subset E$ with $\#(S_i) \leq (D+1)^{\ell-1}$. Then $S = S_1 \cup \ldots \cup S_{D+1}$ is of cardinality at most $(D+1) (D+1)^{\ell-1} = (D+1)^{\ell}$, and thus $\cK(S) \neq \emptyset$. Clearly, $\cK(S) \subset \cK(S_1) \cap \cdots \cap \cK(S_{D+1})$. This finishes the proof of \eqref{eq_1138}. Fix a polynomial $Q$ belonging to the intersection in \eqref{eq_1138}. According to \eqref{eq_1138}, $Q \in \Gamma_{\ell-1}(y, f,M)$. By \eqref{eq_1209}, $Q \in \cK(\emptyset) \subset P + C_T M \cdot \cB_{x,\delta}$, and so $Q - P \in C_T M \cdot \cB_{x,\delta}$, giving \eqref{eq_1208}.
\end{proof}

\begin{lemma}\label{sigma_trans_lem}
If $x,y \in \R^n$, and $\delta \geq |x-y|$, then $\sigma(x) \subset  \sigma(y) +  C_T \cdot \cB_{x,\delta}$.
\end{lemma}
\begin{proof}
Let $P \in \sigma(x)$. Then there exists $\varphi \in C^{m-1,1}(\R^n)$ with $\varphi = 0$ on $E$, $\| \varphi \| \leq 1$, and $J_x \varphi = P$. Let $Q = J_y \varphi$. Then $Q \in \sigma(y)$, and by \eqref{taylor_thm} we have $P - Q \in C_T \cdot \cB_{x,\delta}$.
\end{proof}

\begin{rem}
By \eqref{eq:change_delta}, $\cB_{x,\delta} \subset \delta \cdot \cB_{x,1}$ for $\delta \leq 1$. Therefore, Lemma \ref{sigma_trans_lem} implies the mapping $x \mapsto \sigma(x)$ is continuous, where the space of subsets of $\cP$ carries the topology induced by the Hausdorff metric with respect to any of the topologically equivalent scaled norms.
\end{rem}

\begin{lemma}\label{lemma2}
There exists a constant $C \geq 1$ determined by $m$ and $n$ so that, for any ball $B \subset \R^n$ and $z \in \frac{1}{2}B$, we have
\[
\sigma(z,E \cap B) \cap  \cB_{z,  \diam(B)} \subset C \cdot \sigma(z,E).
\]
 \end{lemma}

\begin{proof}
Choose a cutoff function $\theta \in C^{m-1,1}(\R^n)$ which is supported on $B$, equal to $1$ on $(\frac{1}{2})B$, and satisfies $\| \theta \| \leq C \cdot \delta^{-m}$. Fix $z \in (\frac{1}{2}) B$ and a polynomial $P \in \sigma(z, E \cap B) \cap  \cB_{z,\delta}$. Since $P \in \sigma(z,E \cap B)$ there exists $\varphi \in C^{m-1,1}(\R^n)$ with $\varphi = 0$ on $E \cap B$, $\| \varphi \| \leq 1$, and $J_{z} (\varphi) = P$. Define $\widetilde{\varphi} = \varphi \theta$. This function clearly vanishes on all of $E$. Since $z$ belongs to the ball $(\frac{1}{2})B$ on which $\theta$ is identically $1$, we have $J_{z} ( \widetilde{\varphi}) = J_{z} ( \varphi )= P$. To prove $P \in C  \sigma(z,E)$, all that remains is to establish the seminorm bound $\| \widetilde{\varphi} \| \leq C$. As $\widetilde{\varphi}$ vanishes on $\R^n \setminus B$, it suffices to prove $\| \widetilde{\varphi} \|_{C^{m-1,1}(B)} \leq C$. To do so, we will prove that
\begin{equation}
\label{eq_med}
\begin{aligned}
|J_x(\widetilde{\varphi}) - J_y(\widetilde{\varphi})|_{x,\rho} = | J_x (\varphi) \odot_x J_x (\theta) - & J_y (\varphi) \odot_y J_y (\theta) |_{x,\rho} \leq C \\
& \mbox{for } x,y \in B, \; \rho =  |x-y|.
\end{aligned}
\end{equation}
To prove this estimate we will apply Lemma \ref{lem_prod_cont}. According to \eqref{cutoff_jet},  $J_x(\theta) \in C \delta^{-m} \cB_{x,\delta}$. On the other hand, by \eqref{trans_norm} and the fact $|x-y| \leq \delta$, also $J_y(\theta) \in C \delta^{-m} \cB_{y,\delta} \subset C' \delta^{-m} \cB_{x,\delta}$.  By Taylor's theorem (in the form \eqref{taylor_thm}), $J_x(\theta) - J_y(\theta) \in C \| \theta\| \cB_{x, \rho} \subset C \delta^{-m} \cB_{x,\rho}$.

Note that $|x-z| \leq \delta$, since $x \in B$ and $z \in (\frac{1}{2}) B$. Thus, by Taylor's theorem (see \eqref{taylor_thm}) and  \eqref{trans_norm}, $J_x(\varphi) = (J_x(\varphi) - J_z(\varphi)) + P \in C_T \cB_{x,\delta} + \cB_{z,\delta} \subset C_T \cB_{x,\delta} +   C \cB_{x,\delta} \subset C \cB_{x,\delta}$. On the other hand, by Taylor's theorem, $J_x(\varphi) - J_y(\varphi) \in C_T \cB_{x,\rho}$. We are therefore in a position to apply Lemma \ref{lem_prod_cont} (see Remark \ref{rem1}), with $Q_x$, $Q_y$, $R_x$, and $R_y$ picked to be the jets at $x$ and $y$ of $\varphi$ and $\theta$, respectively. This finishes the proof of \eqref{eq_med}.
\end{proof}

\subsection{Whitney convexity}

Let $E \subset \R^n$ be a finite set. Recall the definition of the sets $\sigma(x) = \sigma(x,E)$ and $\sigma_\ell(x)$, $\ell \geq 0$ (see \eqref{sigma_l}). We now describe an additional important property of the sets $\sigma(x)$ (resp. $\sigma_\ell(x)$) beyond convexity. 

\begin{definition}[Whitney convexity] \label{wc_def} Given a symmetric convex set $\Omega$ in $\cP$, and $x \in \R^n$, the Whitney coefficient of $\Omega$ at $x$ is the infimum over all $R > 0$ such that $(\Omega \cap \cB_{x,\delta}) \odot_x  \cB_{x,\delta} \subset R \delta^m \Omega$ for all $\delta > 0$. Denote the Whitney coefficient of $\Omega$ at $x$ by $w_x(\Omega)$. If no finite $R$ exists, then $w_x(\Omega) = + \infty$. If $w_x(\Omega) < +\infty$ then we say that $\Omega$ is Whitney convex at $x$.
\end{definition}
The term ``Whitney convexity'' was coined by Fefferman \cite{F3}. It is a quantitative analogue of the concept of an ideal. Roughly speaking, if the Whitney coefficient $w_x(\Omega)$ is small then $\Omega$ is ``close'' to an ideal. For example, any ideal $I$ in $\cP_x$ is Whitney convex at $x$ with $w_x(I) = 0$; furthermore, the vanishing of the Whitney coefficient for subspaces provides an equivalent characterization of $\odot_x$-ideals.

We note a few basic properties of Whitney coefficients: For $x \in \R^n$, a symmetric convex set $\Omega \subset \cP$ and $r \geq 1$, it holds that $w_x(r \Omega ) \leq w_x(\Omega)$. If $\Omega_1,\Omega_2 \subset \cP$ are symmetric convex sets then $w_x(\Omega_1 \cap \Omega_2) \leq \max \{ w_x(\Omega_1), w_x( \Omega_2) \}$. Finally, it follows  from \eqref{eq_2115} that $w_x(\Omega) = w_x(\tau_{x,\delta}(\Omega))$ and $w_x(\Omega) = w_{x+h}( T_h \Omega)$ for $\delta > 0$, $h \in \R^n$.

\begin{lemma}\label{sigma_wc_lem}
For any $z \in \R^n$, the sets $\sigma_\ell(z)$ and $\sigma(z)$ are Whitney convex at $z$ with Whitney coefficient at most $C_0$, for a universal constant $C_0=C_0(m,n)$.
\end{lemma}

\begin{proof}
Due to the representation \eqref{sigma_l} and the basic properties of Whitney coefficients stated above, we have $w_x( \sigma_\ell(z) ) \leq \max \{ w_z(\sigma(z,S)) : S \subset E, \; \#(S) \leq (D+1)^\ell\}$. Hence, it suffices to establish the inequality $w_z(\sigma(z,S)) \leq C$ for any subset $S \subset E$ and $z \in \R^n$, where $C$ is a constant determined by $m$ and $n$. Fix $\delta > 0$, and fix arbitrary polynomials $P \in \sigma(z,S) \cap \cB_{z, \delta}$ and $\widetilde{P} \in \cB_{z, \delta}$. We claim that \begin{equation}
 P \odot_z \widetilde{P} \in C \delta^{m} \sigma(z,S).
\label{eq_1448}
\end{equation}
Note that \eqref{eq_1448} implies the inequality $w_z(\sigma(z,S)) \leq C$. Thus, it is sufficient to establish \eqref{eq_1448}.

Since $P \in \sigma(z,S)$, there exists $\varphi \in C^{m-1,1}(\R^n)$ with $\varphi = 0$ on $S$, $J_z (\varphi)  = P$, and $\| \varphi  \| \leq 1$. Let $\theta: \R^n \rightarrow \R$ be a $C^{\infty}$-function, with support contained in the ball $B_\delta(z) := \{ y \in \R^n : |y-z| \leq \frac{\delta}{2}\}$, with $\theta \equiv 1$ in a neighborhood of $z$, and with $\| \theta \| \leq C \delta^{-m}$ for a constant $C$ determined by $m$ and $n$. \footnote{We may obtain such a $\theta$ by rescaling a cutoff function supported on the ball $B_1(z) := \{ y \in \R^n : |y-z| \leq \frac{1}{2}\}$.}

Since $J_z (\theta) = 1$ and $J_z (\varphi) = P$, we have $J_z(\theta  \widetilde{P} \varphi) = 1 \odot_z \widetilde{P} \odot_z P = \widetilde{P} \odot_z P$. To establish \eqref{eq_1448}, it therefore suffices to show that 
\begin{equation}
\label{eq_1449}
J_z(\theta \widetilde{P} \varphi ) \in C \delta^m \sigma(z, S).
\end{equation}
Because $\theta \widetilde{P} \varphi $ vanishes on $S$ (as does $\varphi $), \eqref{eq_1449} is implied by the bound $\| \theta \widetilde{P} \varphi \| \leq C \delta^m$. Because $\theta \widetilde{P} \varphi $ vanishes on $\R^n \setminus B$, it suffices to establish  $\| \theta \widetilde{P} \varphi \|_{C^{m-1,1}(B)} \leq C \delta^m$. To that end, we need to show that
\begin{equation}
\label{eq_1836}
\begin{aligned}
| J_x(\theta) \odot_x \widetilde{P} \odot_x J_x(\varphi) - J_y(\theta) & \odot_y  \widetilde{P} \odot_y J_y(\varphi) |_{x, \rho} \leq C \delta^m, \\
&  \mbox{for} \;  x,y \in B, \; \rho = |x-y|. 
\end{aligned}
\end{equation}
We prepare to apply Lemma \ref{lem_prod_cont} to prove this estimate.

Following the proof of Lemma \ref{lemma2} (using that $J_z(\varphi) = P \in \cB_{z,\delta}$ and  $\diam( \{ x,y,z \}) \leq \delta = \diam(B)$), and by \eqref{cutoff_jet}, the jets $J_x(\varphi)$, $J_y(\varphi)$ belong to $C \cB_{x,\delta}$; and $J_x(\theta)$, $J_y(\theta)$ belong to $ C \delta^{-m} \cB_{x,\delta}$. Furthermore, $\widetilde{P} \in \cB_{z, \delta}$, and hence by \eqref{trans_norm}, $\widetilde{P} \in C \cB_{x,\delta}$. Finally, by Taylor's theorem (rendered as \eqref{taylor_thm}), $J_x(\varphi) - J_y(\varphi) \in C \cB_{x,\rho}$ and $J_x(\theta) - J_y(\theta) \in C \delta^{-m} \cB_{x,\rho}$.

We are in a position to apply Lemma \ref{lem_prod_cont}, with $P_x$, $P_y$, $R_x$, and $R_y$ picked to be the jets at $x$ and $y$ of $\varphi$ and $\theta$, respectively, and with $Q_x = Q_y = \widetilde{P}$. This finishes the proof of the estimate \eqref{eq_1836}, and with it the proof of the lemma.
\end{proof}

\begin{lemma} \label{span-ideal_lem} If $\Omega$ is Whitney convex at $x$, then $\spn(\Omega)$ is an $\odot_x$-ideal in $\cP_x$.
\end{lemma}
\begin{proof}
Choose any $R \in(w_x(\Omega), \infty)$. Then $(\Omega \cap \cB_{x,\delta}) \odot_x  \cB_{x,\delta} \subset R \delta^m  \Omega$ for all $\delta > 0$, and so
\[
\Omega \odot_x \cP_x  = \bigcup_{\delta > 0} (\Omega \cap \cB_{x,\delta}) \odot_x  \cB_{x,\delta} \subset   \bigcup_{\delta > 0}  R \delta^m \Omega = \spn(\Omega).
\]
Thus, $\spn(\Omega)  \odot_x \cP_x = \bigcup_{r >0} r \cdot \Omega \odot_x \cP_x \subset  \spn(\Omega)$, and hence $\spn(\Omega)$ is an $\odot_x$-ideal.
\end{proof}

\subsection{Covering lemmas}

This section contains the covering lemmas that will be used later in the paper. Given a ball $B \subset \R^n$ and $\lambda > 0$, let $\lambda B$ denote the ball with identical center as $B$ and radius equal to $\lambda$ times the radius of $B$.

\subsubsection{Whitney covers}

\begin{definition}\label{Whit_cov_def}
A finite collection $\cW$ of closed balls is a \emph{Whitney cover} of a ball $\widehat{B} \subset \R^n$  if (a) $\cW$ is a cover of $\widehat{B}$, (b) the collection of third-dilates $\{ \frac{1}{3} B : B \in \cW \}$ is pairwise disjoint,  and (c) $ \diam(B_1)/\diam(B_2) \in [1/8,8]$ for all balls $B_1, B_2 \in \cW$ with $\frac{6}{5}B_1 \cap \frac{6}{5}B_2 \neq \emptyset$.
\end{definition}

\begin{lemma}[Bounded overlap] \label{gg_lem}
If  $\cW$ is Whitney cover of $\widehat{B}$ then $\#\{ B \in \cW : x \in \frac{6}{5} B\} \leq 100^n$ for all $x \in \R^n$.
\end{lemma}
\begin{proof}
Let $x \in \R^n$. We may assume $\cW_x := \{ B \in \cW : x \in \frac{6}{5}B\}$ is nonempty, and fix $B_0 \in \cW_x$ of maximal radius. By rescaling, we may assume $\diam(B_0) = 1$. If $B \in \cW_x$ then $\frac{6}{5}B \cap \frac{6}{5} B_0 \neq \emptyset$, and so condition (c) of Definition \ref{Whit_cov_def} implies that  $\diam(B) \in [\frac{1}{8},1]$; thus, by the triangle inequality, $\frac{1}{3} B \subset  (\frac{12}{5} + \frac{1}{3} ) B_0 = \frac{41}{15} B_0$ for all $B \in \cW_x$. Since the collection $\{\frac{1}{3} B \}_{B \in \cW}$ is pairwise disjoint, a volume comparison shows that $\# \cW_x \leq (24 \cdot \frac{41}{15})^n \leq 100^n$.
\end{proof}

\subsubsection{Partitions of unity}

\begin{lemma}[Existence of partitions of unity]\label{pou_lem1}
If $\cW$ is a Whitney cover of $\widehat{B} \subset \R^n$,  then there exist non-negative $C^\infty$ functions $\theta_B : \widehat{B} \rightarrow [0,\infty)$ ($B \in \cW$) such that
\begin{enumerate}
\item $\theta_B = 0$ on $\widehat{B} \setminus \frac{6}{5} B$.
\item $| \partial^\alpha \theta_B (x)| \leq C \diam(B)^{-|\alpha|}$ for all $|\alpha| \leq m$ and $x \in \widehat{B}$.
\item $\sum_{B \in \cW} \theta_B = 1$ on $ \widehat{B}$.
\end{enumerate}
Here, $C$ is a constant determined by $m$ and $n$.
\end{lemma}
\begin{proof}
For $B \in \cW$, let $\psi_B : \R^n \rightarrow \R$ be a $C^\infty$ cutoff function supported on $\frac{6}{5} B$, with $ \psi_B \equiv 1$ on $B$, and with $|\partial^\alpha \psi_B(x)| \leq C \diam(B)^{-|\alpha|}$ for all $x \in \R^n$, $|\alpha| \leq m$. Set $\Psi = \sum_{B \in \cW} \psi_B$ and define
\[
\theta_B(x) := \psi_B(x)/\Psi(x), \quad x \in \widehat{B}.
\]
By definition of a cover, each point in $\widehat{B}$ belongs to some $B \in \cW$; thus, $\Psi \geq 1$ on $\widehat{B}$. Thus $\theta_B \in C^\infty(\widehat{B})$ is well-defined. Property 1 follows because $\psi_B$ is supported on $\frac{6}{5} B$. Furthermore,  $\sum_B \theta_B = \sum_B \psi_B /\Psi = 1$ on $ \widehat{B}$, yielding property 3.

Property 2 is trivial for  $x \in \widehat{B} \setminus \frac{6}{5} B$, as then $J_x(\theta_B) = 0$. Now fix $x \in \frac{6}{5} B \cap \widehat{B}$. If $\psi_{B'}(x) \neq 0$ then $x \in \frac{6}{5} B'$. In particular, $\frac{6}{5} B \cap \frac{6}{5} B' \neq \emptyset$, and hence $\diam(B')/\diam(B) \in [ \frac{1}{8},8 ]$. Furthermore, by Lemma \ref{gg_lem}, the cardinality of  $\cW_x := \{ B' : x \in \frac{6}{5}B' \}$ is at most $100^n$. Hence, 
\begin{equation}\label{brown}
\begin{aligned}
|\partial^\alpha \Psi(x) | & \leq \sum_{B' \in \cW_x} | \partial^\alpha \psi_{B'}(x)| \\
& \leq \sum_{B' \in \cW_x} C \diam(B')^{-|\alpha|} \leq C' \diam(B)^{-|\alpha|} \quad (|\alpha| \leq m).
\end{aligned}
\end{equation}
By a repeated application of the quotient rule for differentiation, and substituting the bounds \eqref{brown} and $| \partial^\alpha \psi_B(x) | \leq C \diam(B)^{- |\alpha|}$,  we conclude that $| \partial^\alpha \theta_B(x) | = | \partial^\alpha (\psi_B/\Psi)(x) | \leq C'' \diam(B)^{-|\alpha|}$ for $|\alpha| \leq m$.
\end{proof}

We mention a few additional properties of the partition of unity $\{\theta_B \}$ in Lemma \ref{pou_lem1}. First, by property 2 of Lemma \ref{pou_lem1} and the definition of the scaled norm $|\cdot |_{x,\delta}$,
\begin{equation}
\label{add_p1}
| J_x (\theta_B) |_{x,\diam(B)} \leq C \diam(B)^{-m} \quad (x \in \widehat{B}).
\end{equation}
By the equivalence of $C^{m-1,1}(\widehat{B})$ and the homogeneous Sobolev space $\dot{W}^{m,\infty}(\widehat{B})$ and by property 2 of Lemma \ref{pou_lem1},
\begin{equation}\label{add_p2}
\| \theta_B \|_{C^{m-1,1}(\widehat{B})} \leq C \max_{|\alpha| = m} \| \partial^\alpha \theta_B \|_{L^\infty(\widehat{B})} \leq C \diam(B)^{-m}.
\end{equation}

\begin{lemma}[Gluing lemma]\label{pou_lem2}
Fix a Whitney cover $\cW$ of $\widehat{B}$, a partition of unity $\{\theta_B\}_{B \in \cW}$ as in Lemma \ref{pou_lem1}, and points $x_B \in \frac{6}{5} B$ for each $B \in \cW$. Suppose $\{F_B\}_{B \in \cW}$ is a collection of functions in $C^{m-1,1}(\R^n)$ with the following properties:
\begin{itemize}
\item $\| F_B \| \leq M_0$.
\item $F_B = f$ on $E \cap \frac{6}{5} B$.
\item $| J_{x_B} F_B - J_{x_{B'}} F_{B'} |_{x_B,\diam(B)} \leq M_0$ whenever $\frac{6}{5} B \cap \frac{6}{5}B' \neq \emptyset$.
\end{itemize}
Let $F = \sum_{B \in \cW} \theta_B F_B$. Then $F \in C^{m-1,1}(\widehat{B})$ with $F=f$ on $E \cap \widehat{B}$ and $\| F \|_{C^{m-1,1}(\widehat{B})} \leq C M_0$, where $C$ is a constant determined by $m$ and $n$.
\end{lemma}
\begin{proof}

The nonzero terms in the sum $F(x) = \sum_B \theta_B(x) F_B(x)$, $x \in E \cap \widehat{B}$, occur when $x \in \frac{6}{5} B$. By assumption,  $F_B(x) = f(x)$ for such $B$. Thus $F(x) = \sum_B \theta_B(x) f(x) = f(x)$. Therefore,  $F = f $ on $E \cap \widehat{B}$. 

We will now bound the seminorm of $F$. We will use the following characterization of the $C^{m-1,1}$ function class: $F \in C^{m-1,1}(\widehat{B})$ if and only if there exists $\epsilon > 0$ and $M \geq 0$ such that $|\partial^{\alpha} F(x) - \partial^\alpha F(y) | \leq M \cdot |x-y|$ for all $x,y \in \widehat{B}$ with $|x-y| \leq \epsilon $ and all multiindices $\alpha$ with $|\alpha|= m-1$. Furthermore, the seminorm $\| F \|_{C^{m-1,1}(\widehat{B})}$ is comparable to the least $M$ as above, up to constant factors depending on $m$ and $n$. This characterization is an easy consequence of the triangle inequality on $\R^n$; we leave the proof as an exercise for the reader. Thus, it suffices to prove that if $|x-y| \leq \frac{1}{100} \delta_{\min}$ for $\delta_{\min} := \min_{B \in \cW}\diam(B)$, then
\begin{equation}\label{eq11}
| J_x F - J_y F|_{x,\rho} \leq C M_0, \mbox{ for } \rho := |x-y|.
\end{equation}

Fix an arbitrary ball $B_0 \in \cW$ with $x \in B_0$. Since $|x-y| \leq \frac{1}{100} \diam(B_0)$, both $x$ and $y$ belong to $\frac{6}{5} B_0$. Note that $\sum_B J_x  \theta_B = \sum_B J_y \theta_B  = 1$. This lets us write
\[
\begin{aligned}
J_x F - J_y F  & =  \sum_{B \in \cW} \biggl[ (J_x F_B - J_x F_{B_0}) \odot_x J_x \theta_B - (J_y F_B - J_y F_{B_0}) \odot_y J_y \theta_B \biggr] \\
& \qquad\qquad + ( J_x F_{B_0} - J_y F_{B_0}).
\end{aligned}
\]
The summands in the main sum on the right-hand side are nonzero only if $x \in \frac{6}{5} B$ or $y \in \frac{6}{5} B$. By Lemma \ref{gg_lem}, there can be at most $2 \cdot 100^n$ many elements $B \in \cW$ with this property. Therefore, to prove inequality \eqref{eq11} it suffices to show that the $|\cdot |_{x,\rho}$ norm of each summand  on the right-hand side is at most $C M_0$. To start, consider the last term and apply Taylor's theorem (in the form \eqref{taylor_thm}):
\[
| J_x F_{B_0} - J_y F_{B_0} |_{x,\rho} \leq C_T  \| F_{B_0} \| \leq C M_0.
\]
Next we select a summand in the main sum by fixing an element $B \in \cW$ with either $x \in \frac{6}{5} B$ or $y \in \frac{6}{5} B$. In either case, $\frac{6}{5} B \cap \frac{6}{5} B_0 \neq \emptyset$. Let $\delta := \diam(B)$. By condition (c) in the definition of a Whitney cover (see Definition \ref{Whit_cov_def}), we have $\delta/\diam(B_0) \in [\frac{1}{8},8]$. Define four polynomials $P_x = J_x(F_B) - J_x(F_{B_0})$ and $R_x = J_x(\theta_B)$, and similarly $P_y=J_y(F_B) - J_y(F_{B_0})$ and $R_y = J_y(\theta_B)$. We will be finished once we show that
\begin{equation}\label{eq12}
| P_x \odot_x R_x - P_y \odot_y R_y|_{x,\rho} \leq C M_0.
\end{equation}

We will  prove \eqref{eq12} using Lemma \ref{lem_prod_cont} (specifically, the form in Remark \ref{rem1}). Let us verify that the hypotheses of this lemma are satisfied. Using $|x-y| = \rho$ and Taylor's theorem (see \eqref{taylor_thm}),
\begin{equation}\label{eq13}
\begin{aligned}
|P_x - P_y|_{x,\rho} &\leq | J_x(F_B) - J_y(F_B)|_{x,\rho} + |J_x(F_{B_0}) - J_y(F_{B_0})|_{x,\rho} \\
&\leq C_T \cdot ( \| F_B \| + \| F_{B_0} \|) \leq  C M_0.
\end{aligned}
\end{equation}
Next write $| P_x |_{x,\delta} \leq |P_{x_{B_0}} - P_x |_{x,\delta} + |P_{x_{B_0}}|_{x,\delta}$. As $x \in B_0$ and $x_{B_0} \in \frac{6}{5}B_0$, we have $|x-x_{B_0}| \leq \frac{6}{5} \diam(B_0) \leq 3 \delta$. Thus, by \eqref{eq:change_delta} and following the proof of \eqref{eq13}, $|P_{x_{B_0}} - P_x |_{x,\delta} \leq 3^m |P_{x_{B_0}} - P_x |_{x,3 \delta} \leq C' M_0$. Then by \eqref{eq:change_delta} and \eqref{trans_norm}, the hypothesis in the third bullet point of this lemma, and another application of Taylor's theorem, 
\[
\begin{aligned}
|P_{x_{B_0}}|_{x,\delta} &\leq | J_{x_B} (F_B) - J_{x_{B_0}}(F_{B_0})|_{x,\delta} + |J_{x_B} (F_B) - J_{x_{B_0}}(F_B)|_{x,\delta}   \\
&\leq C | J_{x_B} (F_B) - J_{x_{B_0}}(F_{B_0})|_{x_{B_0},\delta} + C |J_{x_B} (F_B) - J_{x_{B_0}}(F_B)|_{x_{B_0},\delta} \\
& \leq C' | J_{x_B} (F_B) - J_{x_{B_0}}(F_{B_0})|_{x_{B_0},\diam(B_0)} + C' |J_{x_B} (F_B) - J_{x_{B_0}}(F_B)|_{x_{B_0},4\delta} \\
&\leq C'' M_0.
\end{aligned} 
\]
Here, note we are using that $|x_B - x_{B_0}| \leq \frac{6}{5} \diam(B) + \frac{6}{5} \diam(B_0) \leq 4 \delta$ in the final application of Taylor's theorem. In conclusion, $|P_x|_{x,\delta} \leq C M_0$. By the identical argument, $|P_y|_{y,\delta} \leq C M_0$ -- then by \eqref{trans_norm},  $|P_y|_{x,\delta} \leq C' M_0$.

Next, note the estimate $|R_x - R_y|_{x,\rho} \leq C \delta^{-m}$ is a direct consequence of Taylor's theorem and \eqref{add_p2}. Also, $|R_x|_{x,\delta} \leq C \delta^{-m}$ is a direct consequence of \eqref{add_p1}. Similarly,  $|R_y |_{y,\delta} \leq C \delta^{-m}$, and thus by \eqref{trans_norm},  $|R_y|_{x,\delta} \leq C' \delta^{-m}$.

We obtain \eqref{eq12} by an application of Lemma \ref{lem_prod_cont} (see Remark \ref{rem1}), which finishes the proof of the lemma.
\end{proof}

\section{Transversality}

Let $(X,\langle \cdot, \cdot \rangle)$ be a real Hilbert space of finite dimension $d := \dim X < \infty$. We denote the norm of $X$ by $|\cdot| = \sqrt{ \langle \cdot,\cdot \rangle}$, and let $\cB$ be the unit ball of $X$. Let $\cS$ be the set of closed symmetric convex subsets of $X$, and let $d_H : \cS \times \cS \rightarrow [0,\infty]$ be the Hausdorff metric, namely,
\[
d_H(\Omega_1,\Omega_2) := \inf \{ \epsilon > 0 : \Omega_1 \subset \Omega_2 + \epsilon \cB, \; \Omega_2 \subset \Omega_1 + \epsilon \cB \}.
\]
Given a set $A \subset X$ and subspace $V \subset X$, let $A/V$ (the \emph{quotient} of $A$ by $V$) be the image of $A$ under the quotient mapping $\pi : X \rightarrow X/V$, i.e., $A/V := \{ a + V : a \in A \}$.
\begin{definition}
Let $V$ be a linear subspace of $X$, let $\Omega \in \cS$, and let $R \geq 1$. We say that $\Omega$ is $R$-transverse to $V$ if (1) $\cB/V \subset R \cdot (\Omega \cap \cB)/V$, and (2) $\Omega \cap V \subset R \cdot \cB$.
\end{definition}


\begin{lemma}[Stability I] \label{lem000} If $\Omega$ is $R$-transverse to $V$, then $\Omega + \lambda \cB$ is $(R + 3 R^2 \lambda)$-transverse to $V$ for any $\lambda > 0$.
\end{lemma}
\begin{proof}
Suppose $\Omega$ is $R$-transverse to $V$. By condition (1) in the definition of transversality, we have 
\[
\cB /V \subset R \cdot  ( \Omega \cap \cB )/V \subset R \cdot \left( (\Omega + \lambda\cB)  \cap \cB \right)/V.
\]
 All that remains is to show
\[
(\Omega + \lambda\cB ) \cap V \subset (R + 3 R^2 \lambda) \cB.
\]
Fix $P \in (\Omega + \lambda \cB) \cap V$. Write $P = P_0 + P_1$ with $P_0 \in  \Omega$ and $P_1 \in \lambda \cB$. By condition (1) in the definition of transversality, we have $\lambda \cB/V \subset R \lambda(\Omega \cap \cB)/V$. Since $P_1 \in \lambda \cB$, there exists a polynomial $P_2 \in R \lambda (\Omega \cap \cB)$ with $P_1/V = P_2/V$ -- or rather, $P_1 - P_2 \in V$. Define $\tilde{P} := P - (P_1 - P_2) \in V$. Then $\tilde{P} = P_0 + P_2$. Because $P_0 \in \Omega$ and $P_2 \in R \lambda \cdot  \Omega$, we have $\tilde{P} \in (R\lambda + 1) \cdot ( \Omega \cap V) \subset (R \lambda + 1) \cdot  R \cB$, where the second inclusion uses condition (2) in the definition of transversality. Therefore,
\[
P =  \tilde{P} + P_1 - P_2 \in (R \lambda + 1) R  \cB + \lambda \cB + R \lambda \cB \subset (R^2 \lambda + R + \lambda + R \lambda) \cB.
\]
We conclude that $P \in (R + 3 R^2 \lambda) \cB$, which completes the proof of the lemma.

\end{proof}

\begin{lemma} [Stability II] \label{comp_lem2}
Let $\Omega_1, \Omega_2 \in \cS$, and let $R \geq 1$, $\tilde{R} \geq 4 R$. If $\Omega_1$ is $R$-transverse to $V$, then the following holds:
\begin{itemize}
\item If $d_H(\Omega_1,\Omega_2) \leq \frac{1}{4 R}$ then $\Omega_2$ is $4R$-transverse to $V$.
\item If $d_H(\Omega_1 \cap \tilde{R} \cB,\Omega_2 \cap \tilde{R} \cB ) \leq \frac{1}{4 R}$ then $\Omega_2$ is $4R$-transverse to $V$.
\end{itemize}
\end{lemma}
\begin{proof}
For the proof of the first bullet point, we may suppose $\Omega_1 \subset \Omega_2 + \lambda \cB$ and $\Omega_2 \subset \Omega_1 + \lambda \cB$ for $\lambda = \frac{1}{3 R}$. According to Lemma \ref{lem000}, $\Omega_1 + \lambda \cB$ is $2R$-transverse to $V$. Thus,
\begin{equation}\label{ee1}
\Omega_2 \cap V \subset \left(\Omega_1 + \lambda  \cB \right) \cap V \subset 2R \cdot  \cB.
\end{equation}
Also,
\[
 \cB/V \subset R \cdot \left( \Omega_1 \cap  \cB \right)/V \subset R \cdot \left( \left(\Omega_2 + \lambda  \cB \right) \cap  \cB \right)/V.
\]
By \eqref{fact1}, $(\Omega_2 + \lambda \cB ) \cap  \cB \subset ( \Omega_2 \cap 2  \cB) + \lambda \cB$, hence,
\[
 \cB/V  \subset R \cdot ( \Omega_2 \cap 2  \cB + \lambda  \cB )/V = R \cdot ( \Omega_2 \cap 2  \cB)/V + R \lambda \cdot  \cB /V.
\]
Recall $R \lambda = \frac{1}{3}$, hence $K \subset T + K/3$ for $K =  \cB/V$ and $T = R \cdot  ( \Omega_2 \cap 2  \cB)/V$. From \eqref{fact2} we conclude that $K \subset 2T$, i.e.,
\begin{equation}\label{ee2}
 \cB/V \subset 2 R \cdot ( \Omega_2 \cap 2  \cB  )/V \subset 4R \cdot (\Omega_2 \cap  \cB )/V.
\end{equation}
From \eqref{ee1} and \eqref{ee2} we conclude that $\Omega_2$ is $4R$-transverse to $V$.

Note $\Omega_1$ is $R$-transverse to $V$ iff $\Omega_1 \cap \tilde{R} \cB$ is $R$-transverse to $V$ (since $\tilde{R} \geq R$), and similarly, $\Omega_2$ is $4R$-transverse to $V$ iff $\Omega_2 \cap \tilde{R}  \cB$ is $4R$-transverse to $V$ (since $\tilde{R} \geq 4R$). Thus, by applying the first bullet point to the sets $\Omega_1 \cap \tilde{R} \cB$ and $\Omega_2 \cap \tilde{R}  \cB$, we obtain the conclusion in the second bullet point.
\end{proof}

\begin{lemma}[Stability III]\label{lem01}
Suppose $\Omega$ is $R$-transverse to $V$, and let $U : X \rightarrow X$ be a unitary transformation. Then $U(\Omega)$ is $R$-transverse to $U(V)$. If additionally $\| U  - id \|_{op} \leq \frac{1}{16R^2}$, then $U(\Omega)$ is $4R$-transverse to $V$ and $\Omega$ is $4R$-transverse to $U(V)$.
\end{lemma}
\begin{proof}
Unitary transformations preserve the metric structure of $X$, and in particular, they preserve transversality. If $\| U - id \|_{op} \leq  \frac{1}{16R^2}$ then
\[
d_H(\Omega \cap 4R \cB, U(\Omega) \cap 4R \cB )  = d_H(\Omega \cap 4R \cB, U(\Omega \cap 4R \cB))  \leq  \| U - id \|_{op}  \cdot 4R \leq \frac{1}{4 R}.
\]
Therefore, by Lemma \ref{comp_lem2}, $U(\Omega)$ is $4R$-transverse to $V$. Similarly, $U^{-1}(\Omega)$ is $4R$-transverse to $V$, and thus by the first claim we have that $\Omega$ is $4R$-transverse to $U(V)$.
\end{proof}

We also prove a version of Lemma \ref{lem000} in which the upper and lower inclusions on $\Omega$ involve two different constants. 

\begin{lemma}[Stability IV] \label{lem:stabv} Let $R,Z \geq 1$ and $\lambda \geq 1$ be given. If $\Omega$ is a symmetric closed convex set in a Hilbert space $X$, and $V \subset X$ is a subspace, satisfying (i) $ \cB/V \subset R \cdot (\Omega \cap \cB)/V$ and  (ii) $\Omega \cap V \subset Z \cB$, then
\begin{equation}\label{eqn:201}
(\Omega + \lambda \cB) \cap V \subset Z \cdot (3R \lambda + 1) \cB.
\end{equation}
\end{lemma}
\begin{proof}
To prove \eqref{eqn:201}, we copy the proof of Lemma \ref{lem000}; wherever we applied conditions (1) or (2) in the definition of transversality, we instead apply (i) or (ii).
\end{proof}


\subsection{Transversality in the space of polynomials}

\begin{definition}\label{trans_P_defn}
Given a closed, symmetric, convex set  $\Omega \subset \cP$, a subspace $V \subset \cP$, $R \geq 1$, $x \in \R^n$, and $\delta > 0$, we say that $\Omega$ is $(x,\delta, R)$-transverse to $V$ if $\Omega$ is $R$-transverse to $V$ with respect to the Hilbert space structure $(\cP, \langle \cdot,\cdot \rangle_{x,\delta})$, i.e., (1) $\cB_{x,\delta} / V \subset R \cdot (\Omega \cap \cB_{x,\delta})/ V$, and (2) $\Omega \cap V \subset R \cdot \cB_{x,\delta}$.
\end{definition}

Our next result establishes a few basic properties of transversality in this setting.
\begin{lemma}\label{lem00} If $\Omega$ is $(x,\delta,R)$-transverse to $V$, then the following holds:
\begin{itemize}
\item  $T_h \Omega$ is  $(x+h,\delta,R)$-transverse to $T_h V$.
\item $\tau_{x,r} \Omega$ is $(x,\delta/r,R)$-transverse to $\tau_{x,r} V$.
\item If $\delta' \in [ \kappa^{-1} \delta, \kappa \delta]$ for some $\kappa \geq 1$, then $\Omega$ is $(x,\delta', \kappa^{m} R)$-transverse to $V$.
\end{itemize}
\end{lemma}
\begin{proof}
The proof of the first and second bullet points is easy: Apply $T_h$ and $\tau_{x,r}$ to both sides of (1) and (2) in Definition \ref{trans_P_defn}, and use the identities $T_h \cB_{x,\delta} = \cB_{x+h,\delta}$ and $\tau_{x,r} \cB_{x,\delta} = \cB_{x,\delta/r}$. The third bullet point follows from the equivalence of the unit balls $\cB_{x,\delta} \subset \max\left\{1, \left(\delta/\delta' \right)^m \right\} \cdot \cB_{x,\delta'}$ and $\cB_{x,\delta'} \subset \max\left\{1, \left(\delta'/\delta \right)^m \right\} \cdot \cB_{x,\delta}$, as well as the property that $A \cap (r\cdot B) \subset r \cdot (A \cap B)$ if $A$ and $B$ are symmetric convex sets, and $r \geq 1$.

\end{proof}

The continuity of the mapping $x \mapsto \sigma(x)$ can be used to show that the transversality of the set $\sigma(x)$ with respect to a fixed subspace is stable with respect to small perturbations of the basepoint.

\begin{lemma}\label{trans_lem}
There exists $c_1 = c_1(m,n) > 0$ so that the following holds. Let $V \subset \cP$ be a subspace, $x,y\in  \R^n$, $\delta > 0$, $R \geq 1$.  Suppose that $\sigma(x)$ is $(x,\delta,R)$-transverse to $V$ and $|x-y| \leq c_1 \frac{\delta}{R}$. Then $\sigma(y)$ is $(y,\delta, 8 R)$-transverse to $V$.
\end{lemma}
\begin{proof}
If $c_1 < \frac{1}{4 C_T}$, where $C_T$ is the constant in \eqref{taylor_thm}, then by Lemma \ref{sigma_trans_lem},
\[
\sigma(y) \subset \sigma(x) + C_T \cdot \cB_{ x, c_1 \cdot \frac{\delta}{R}} \subset \sigma(x) + C_T  \cdot \biggl(\frac{c_1}{R} \biggr)  \cdot \cB_{x, \delta} \subset  \sigma(x) + \biggl(\frac{1}{4R} \biggr) \cdot \cB_{x,\delta}. 
\]
Similarly, $\sigma(x) \subset \sigma(y) + (\frac{1}{4R}) \cdot  \cB_{x,\delta}$. Thus, $d_H^{x,\delta}(\sigma(x),\sigma(y)) \leq \frac{1}{4R}$, where $d_H^{x,\delta}$ is the Hausdorff distance with respect to the norm $|\cdot|_{x,\delta}$ on $\cP$. From Lemma \ref{comp_lem2} we conclude that $\sigma(y)$ is  $(x,\delta,4R)$-transverse to $V$. Since $|x-y| \leq c_1  \delta/R \leq c_1 \delta$, if $c_1$ is sufficiently small then $(\frac{9}{10}) \cdot \cB_{y,\delta} \subset \cB_{x,\delta} \subset (\frac{10}{9}) \cdot \cB_{y,\delta}$. Therefore we can replace $\cB_{x,\delta}$ by $\cB_{y, \delta}$ in the definition of transversality, at the cost of increasing the constant $4R$ to $8R$. Thus, $\sigma(y)$ is  $(y,\delta,8R)$-transverse to $V$.
\end{proof}

\subsection{Ideals in the ring of polynomials and DTI subspaces}

\begin{definition}  \label{dti_defn}
A  subspace $V \subset \cP$ is translation-invariant if $T_{h} V = V$ for all $h \in \R^n$, and $V$ is dilation-invariant at $x \in \R^n$ if $\tau_{x,\delta}V = V$ for all $\delta > 0$. Say that $V$ is dilation-and-translation-invariant (DTI) if $T_h \tau_{x,\delta} V = V$ for all $x, h \in \R^n$, $\delta > 0$. We write $\dti$ to denote the collection of all DTI subspaces of $\cP$.
\end{definition}

\begin{rem}
Equivalently, $V \subset \cP$ is translation-invariant if $P \in V, Q \in \cP \implies Q(\partial) P \in  V$.  Since $T_h = \tau_{(1-\delta)^{-1} h,\delta^{-1}} \circ \tau_{0,\delta}$ (for any $\delta > 1$), any translation operator is a composition of dilation operators. Thus, $V$ is DTI if and only if $\tau_{x,\delta} V = V$ for all $(x,\delta) \in \R^n \times (0,\infty)$. 
\end{rem}

We now illustrate a connection between translation-invariant subspaces and ideals in $\cP_x$.

\begin{lemma}
\label{dti_trans_lem}
Let $(x,\delta) \in \R^n \times (0,\infty)$. Let $V^\perp$ be the orthogonal complement of a subspace $V \subset \cP$ with respect to the inner product $\langle \cdot,\cdot \rangle_{x,\delta}$. Then $V$ is translation-invariant if and only if $V^\perp$ is an $\odot_x$-ideal in $\cP_x$.
\end{lemma}

\begin{proof}

Translating, we may assume that $x = 0$. Rescaling preserves the property of $V$ being translation-invariant, and also of $V^{\perp}$ being an $\odot_x$-ideal, according to \eqref{eq_2115}. Hence we may assume that $\delta = 1$. Note the identity $\langle Q,P \rangle = Q(\partial) (P)(0)$ for any $P,Q \in \cP$. Note $\partial^\alpha$ annihilates $\cP$ for $|\alpha| \geq m$, and hence $R(\partial) [ Q(\partial) P] = (R \odot_0 Q)(\partial) P$ for any $P,Q,R \in \cP$. Suppose that $V$ is a translation-invariant subspace, and let $Q \in V^{\perp}$. Then, for any $h \in \R^n$ and $P \in V$, also $T_h P \in V$ and hence,
\[
0 = \langle Q, T_h (P) \rangle  = Q(\partial) \left[ T_h (P) \right] (0) = T_h( Q(\partial) P )(0) = Q(\partial) P (-h).
\] 
Consequently, $Q(\partial) P = 0$. Thus, for any $R \in \cP$, we have $(R \odot_0 Q) (\partial) P = R(\partial) \left[ Q(\partial) P  \right] = 0$. In particular, $\langle R \odot_0 Q, P \rangle = 0$ for any $P \in V$ and hence $R \odot_0 Q \in V^{\perp}$. This shows that $V^{\perp}$ is an $\odot_0$-ideal.

For the other direction, suppose that $V^{\perp}$ is an $\odot_0$-ideal. Let $P \in V$ and $R \in \cP$. Then for any $Q \in V^{\perp}$,
\[
0 = \langle R \odot_0 Q, P \rangle = Q(\partial)\left[ R(\partial)P \right](0) =  \langle Q, R(\partial) P \rangle. 
\]
This means that $R(\partial) P \in (V^{\perp})^{\perp} = V$. Hence $R(\partial) P \in V$ whenever $P \in V$ and $R \in \cP$, and consequently the subspace $V$ is translation-invariant.
\end{proof}

We say that two subspaces $V_1,V_2 \subset \cP$ are \emph{complementary} if $V_1 + V_2 = \cP$ and $V_1 \cap V_2 = \{0\}$.

\begin{lemma}\label{dti_comp_lem}
For any $\odot_0$-ideal $I$ in $\cP_0$,  there exists $V \in \dti$ that is complementary to $I$. 
\end{lemma}
\begin{proof}
Set $I_* = \lim_{\delta \rightarrow 0} \tau_{0,\delta} (I)$ (where the Grassmanian is endowed with the usual topology). Let us first show that this limit exists: Consider the canonical projection $\pi_k : \cP_0 \rightarrow \cP_0^k$ onto the subspace of $k$-homogeneous polynomials $\cP_0^k := \spn \{ z^\alpha : |\alpha| = k\}$, and denote the subspace of $(\geq k)$-homogeneous polynomials $\cP_0^{\geq k} := \spn \{ z^\alpha : |\alpha| \geq k\}$. By Gaussian elimination we can pick a basis $\cB_1 := \{ P^k_j \}^{0 \leq k \leq m-1}_{1 \leq j \leq N_k}$ for $I$ in the \emph{block form}: $P^k_j \in \cP^{\geq k}_0$, and $\cB_0 := \{ \pi_k P^k_j \}^{0 \leq k \leq m-1}_{1 \leq j \leq N_k}$ is linearly independent in $\cP_0$. The family $\cB_\delta := \{ \delta^{m-k} \tau_{0,\delta} (P_j^k) \}_{k,j}$ converges elementwise as $\delta \rightarrow 0$ to $\cB_0$. Since $\cB_\delta$ is a basis for  $\tau_{0,\delta}(I)$, and $\cB_0$ is a basis for $I_* := \spn (\cB_0)$, we learn that $\tau_{0,\delta}(I)$ converges to $I_*$, as desired. 

The ideals form a closed subset of the Grassmanian, thus $I_*$ is an ideal in the ring $\cP_0$. Let $V$ be the orthogonal complement of $I_*$ with respect to the standard inner product on $\cP_0$. Observe that $I_*$ is dilation-invariant at $x=0$, i.e., $\tau_{0,\delta} I_* = I_*$ for all $\delta > 0$. Equivalently, $I_*$ is a direct sum of homogeneous subspaces of $\cP_0$, i.e., $I_* = I^0 + \cdots + I^{m-1}$, with $I^k \subset \cP_0^k$. But then $V$ is also a direct sum of homogeneous subspaces of $\cP_0$, and so $V$ is dilation-invariant at $x=0$. From Lemma \ref{dti_trans_lem}, we also know that $V$ is translation-invariant.  Thus, $V \in \dti$. The subspaces $I_*$ and $V$ are complementary and this property is open in $\cG \times \cG$. By definition of $I_*$ as a limit, $\tau_{0,\delta}(I)$ and $V$ are complementary for some $\delta > 0$. By an application of the isomorphism of vector spaces $\tau_{0,\delta^{-1}}$, we learn that $I$ and $\tau_{0,\delta^{-1}} V$ are complementary. To finish the proof, recall that $V \in \dti$, and hence $\tau_{0,\delta^{-1}} V=V$.
\end{proof}

Our next result says that every Whitney convex set is transverse to a DTI subspace.

\begin{lemma}\label{label_lem}
Given $A \in [1,\infty)$, there exists a constant $R_0 = R_0(A,m,n)$ so that the following holds. Let $\Omega$ be a closed, symmetric, convex subset of $\cP$. If $\Omega$ is Whitney convex at $x \in \R^n$ with $w_x(\Omega) \leq A$, and $\delta > 0$, then there exists $V \in \dti$ such that $\Omega$ is $(x,\delta,R_0)$-transverse to $V$.
\end{lemma}
\begin{proof}

By the second bullet point in Lemma \ref{lem00}, $\Omega$ is $(x,\delta,R)$-transverse to $V$ if and only if $\tau_{x,\delta} \Omega$ is $(x,1,R)$-transverse to $\tau_{x,\delta}V$. Thus, by the remark following Definition \ref{wc_def}, we may rescale and assume that $\delta = 1$. Similarly, by translating we may assume that $x=0$.

Let $\cS$ be the set of closed, symmetric, convex subsets of $\cP$. We endow $\cS$ with the topology of local Hausdorff convergence, i.e., $\Omega_j \rightarrow \Omega$ iff $\lim_{j \rightarrow \infty} d_H(\Omega_j \cap R \cB,\Omega \cap R \cB) = 0$ for all $R > 0$ -- here, $\cB \subset \cP$ is the unit ball with respect to the norm $|\cdot| = |\cdot|_{0,1}$ on $\cP$, and $d_H$ is the Hausdorff metric with respect to this norm. As a consequence of the Blaschke selection theorem, thus endowed, $\cS$ is a compact space. Write $\cG$ to denote the Grassmanian of all subspaces of $\cP$, and $\cG_k \subset \cG$ the Grassmanian of all $k$-dimensional subspaces. We may identify $\cG$ as a compact subspace of $\cS$.

For any $(x,\delta) \in \R^n \times (0,\infty)$, the isomorphism $\tau_{x,\delta} : \cP \rightarrow \cP$ induces a continuous mapping on the Grassmanian $\tau_{x,\delta} : \cG \rightarrow \cG$.  Thus, $\dti = \{ V \in \cG : \tau_{x,\delta} V = V \;\; \forall (x,\delta) \in \R^n \times (0,\infty) \}$ is a closed subset of $\cG$, and hence $\dti$ is compact.

The conclusion of the lemma is equivalent to the existence of a constant $R_0=R_0(A,m,n)$ so that $\phi(\Omega) \leq R_0$ for all $\Omega \in wc_A$, where
\[
\begin{aligned}
& wc_A := \{ \Omega \in \cS : \Omega \mbox{ is Whitney convex at } 0 \mbox{ with } w_0(\Omega) \leq A \}, \\
& \phi: wc_A \rightarrow [0,\infty], \;\; \phi(\Omega) := \inf \{ \psi(\Omega,V) : V \in \dti \}, \;\; \mbox{with} \\
& \psi : wc_A \times \dti \rightarrow [0,\infty], \;\; \mbox{where} \\
& \qquad \psi(\Omega,V) := \inf \{ R :  \Omega \cap V \subset R \cdot \cB, \;\; \cB/V \subset  R  \cdot (\Omega \cap \cB)/V \}.
\end{aligned}
\]
If $\Omega_n \rightarrow \Omega$, $\Omega_n \in wc_A$, $\delta > 0$, and $A^* > A$, then 
\[
(\Omega \cap \cB_{0,\delta}) \odot_0 \cB_{0,\delta} = \lim_{n \rightarrow \infty} (\Omega_n \cap \cB_{0,\delta}) \odot_0 \cB_{0,\delta} \subset \lim_{n \rightarrow \infty} A^* \delta^m \Omega_n = A^* \delta^m \Omega,
\]
where we used the continuity of $\odot_0$ on $\cS \times \cS$. So $wc_A$ is closed, and hence compact. We claim that $\psi$ is upper semicontinuous (usc). Indeed, $\psi = \inf_{R > 0} \psi_R$, with $\psi_R = R 1_{E_R} + \infty 1_{E_R^c}$ and 
\[
E_R = \{(\Omega,V) \in \cS \times \dti : \exists R' < R, \; \Omega \cap V \subset R' \cdot \cB \mbox{ and } \cB/V \subset R' \cdot (\Omega \cap \cB) / V \}.
\]
As $E_R$ is open, $\psi_R$ is usc. Hence the same is true of $\psi$, and also of $\phi$.

Since $\phi$ is usc and $wc_A$ is compact, it suffices to show that $\phi(\Omega) < \infty$ for all $\Omega \in wc_A$. Since $\Omega$ is Whitney convex at $0$, $I = \spn (\Omega)$ is an ideal in $\cP_0$ (see Lemma \ref{span-ideal_lem}). By Lemma \ref{dti_comp_lem} there exists a subspace $V \in \dti$ which is complementary to $I$, i.e., $V \cap I = \{0\}$ and $V + I  = \cP$. Note that $\spn (\Omega + V) = I + V = \cP$, and so by convexity, $\Omega + V$ contains a ball $\epsilon \cB$ for some $\epsilon > 0$. If $\epsilon \cB \subset \Omega + V$, it follows that $\epsilon \cB / V \subset \Omega / V$. Thus,
\[
\epsilon \cB / V \subset \bigcup_{R > 0} (\Omega \cap R \cB)/V.
\]
By compactness, there  exists an  $R > 0$ with $\frac{\epsilon}{2} \cB/V \subset (\Omega \cap R \cB)/V \subset R ( \Omega \cap \cB)/V$. Thus, $ \cB/V \subset \frac{2R}{\epsilon}  ( \Omega \cap \cB)/V$. Combined with $V \cap \Omega \subset V \cap I = \{0\}$, this implies that $\phi(\Omega) \leq \frac{2R}{\epsilon}$.
\end{proof}

For any $x \in \R^n$, the set $\sigma(x) = \sigma(x,E)$ is Whitney convex at $x$ with $w_x(\sigma(x)) \leq C_0$ (see Lemma \ref{sigma_wc_lem}). Let $R_0$ be the constant from Lemma \ref{label_lem} with $A = C_0$. Then
\begin{equation}
\label{label1}
\left\{
\begin{aligned}
&\mbox{for any finite set} \; E \subset \R^n,\; \mbox{for any } (x,\delta) \in \R^n \times (0,\infty),\\
&\mbox{there exists } V \in \dti \mbox{ such that } \sigma(x) \mbox{ is } (x,\delta,R_0) \mbox{-transverse to } V.
\end{aligned}
\right.
\end{equation}

\textbf{Constants:} Recall the constant  $c_1$ is defined in Lemma \ref{trans_lem}. We specify constants $R_\lbl \ll R_\med \ll R_\big \ll R_\huge$, $C_*$, and $C_{**}$, defined as follows:
\begin{equation}
\label{fix_c}
\left\{
\begin{aligned}
&R_\lbl := 8 R_0,  \; R_\med := 256D R_\lbl , \; R_\big := 10^{m} R_\med, \;  R_\huge := 2^{m+3} R_\big  \\
& C_* := 20 c_1^{-1} R_\big, \; C_{**} =1 + 2^m C_T (1 +  R_\lbl (5C_*)^m).
\end{aligned}
\right.
\end{equation}

\begin{lemma}\label{main-labeling-lemma}
Let $B$ be a closed ball in $\R^n$. There exists $V \in \dti$ such that $\sigma(z)$ is $(z, C_* \diam(B),  R_{\lbl})$-transverse to $V$ for all $z \in 100 B$.
\end{lemma}
\begin{proof}
Let $x_0 \in \R^n$ be the center of $B$. We apply \eqref{label1} with $x=x_0$ and $\delta = C_*  \diam(B)$. Thus, $\sigma(x_0)$ is $(x_0, C_* \diam(B), R_0)$-transverse to some $V \in \dti$. Let $z \in 100B$ be arbitrary. Then $|z-x_0| \leq 100 \diam(B) \leq c_1 \frac{C_* \diam(B)}{R_0}$ (see \eqref{fix_c}). By Lemma \ref{trans_lem}, we conclude that $\sigma(z)$ is $(z, C_* \diam(B), 8R_0)$-transverse to $V$.

\end{proof}

\section{Complexity} \label{comp_sec}

The left and right endpoints of an interval $I \subset \R$ are denoted by $l(I)$ and $r(I)$, respectively. An interval $J$ is \emph{to the left} of an interval $I$, written $J < I$, if either $r(J) < l(I)$ or $r(J) = l(I)$ and $l(J) < l(I)$. Let $X$ be a finite-dimensional Hilbert space with inner product $\langle \cdot, \cdot \rangle_X$, set $d := \dim X < \infty$, and denote the norm and unit ball of $X$ by $| \cdot |_X = \sqrt{ \langle \cdot,\cdot \rangle}_X$ and $\cB = \{ x \in X : |x|_X \leq 1 \}$. Let $\Psi : \R^D \rightarrow X$ be a coordinate transformation of the form $\Psi(v) = \sum_j v_j e_j$ for an orthonormal basis $\{ e_j \}_{1 \leq j \leq d}$ of $X$. Fix $\vec{m} = (m_1,\cdots,m_d) \in \Z_{\geq 0}^d$ and a $1$-parameter family of maps $T_\delta : X \rightarrow X$ ($\delta > 0$) of the form $T_\delta = \Psi \widetilde{T}_\delta \Psi^{-1}$, where the transformation $\widetilde{T}_\delta : \R^d \rightarrow \R^d$ is represented in standard Euclidean coordinates by a diagonal matrix $D_\delta = \mbox{diag}(\delta^{-m_1},\cdots,\delta^{-m_d})$. 

\begin{definition}\label{comp_defn}
Given a closed, symmetric, convex set $\Omega \subset X$, the complexity of $\Omega$ relative to the dynamical system $\cX = (X, T_\delta)_{\delta>0}$ at scale $\delta_0 > 0$ with parameter $R \geq 1$-- written  $\cC_{\cX,\delta_0,R}(\Omega)$ -- is the largest integer $K \geq 1$ such that there exist intervals $I_1 > I_2 > \cdots > I_K$ in $(0,\delta_0]$ and subspaces $V_1,V_2\cdots,V_K \subset X$, such that $T_{r(I_k)} (\Omega)$ is $R$-transverse to $V_k$, but $T_{l(I_k)}(\Omega)$ is not $256dR$-transverse to $V_k$  for all $k=1,\cdots,K$. If no such $K$ exists, let $\cC_{\cX,\delta_0,R}(\Omega) := 0$.
\end{definition}

\begin{proposition}\label{bdd_comp_prop}
Given $R \geq 1$ and $\vec{m} \in \Z^d_{\geq 0}$, there exists a constant $K_0 = K_0(d, \vec{m},R)$ such that $\cC_{\cX,\delta_0,R}(\Omega) \leq K_0$ for all closed, symmetric, convex sets $\Omega \subset X$ and all $\delta_0>0$.
\end{proposition}

\subsection{Background on semialgebraic geometry}

We review some standard terminology from semialgebraic geometry: A set $B \subset \R^d$ is a \emph{basic set} if it is the solution set of a finite number of polynomial inequalities, i.e.,  $B = \{ x \in \R^d : p_i(x) \leq 0, \; q_j(x) < 0 \; \forall i \forall j \}$, for polynomials $p_1,\cdots,p_k, q_1,\cdots,q_l$ on $\R^d$. A \emph{semialgebraic set} is a finite union of basic sets. The class of semialgebraic sets is obviously closed under finite unions/intersections and complements. The celebrated Tarski-Seidenberg theorem on quantifier elimination implies that the class of semialgebraic sets is closed under projections $\pi : \R^d \rightarrow \R^{d-1}$; see \cite{vdD}. Semialgebraic sets are closely related to \emph{first-order formulas} over the reals, which are defined by the following elementary rules: (1) If $p$ is a polynomial on $\R^d$, then ``$p \leq 0$'' and ``$p < 0$'' are (first-order) formulas, (2) If $\Phi$ and $\Psi$ are formulas, then ``$\Phi$ and $\Psi$'', ``$\Phi$ or $\Psi$'', and ``not $\Phi$'' are formulas, and (3) If $\Phi$ is a formula and $x$ is a variable of $\Phi$ (ranging in $\R$), then ``$\exists x \; \Phi$'' and ``$\forall x \; \Phi$'' are formulas. A first-order formula is \emph{quantifier-free} if it arises only via (1) and (2). Clearly the semialgebraic sets are precisely the solution sets of quantifier-free formulas. The Tarski-Seidenberg theorem states that every first-order formula is equivalent (i.e., has an identical solution set) to a quantifier-free formula. Accordingly, the solution set of a first-order formula is semialgebraic. In particular, the set $\cM^+$ of all positive-definite $d \times d$ matrices is a semialgebraic subset of $\R^{d \times d}$ because it can be represented as the solution set of a first-order formula: $\cM^+ = \{ (a_{ij})_{1 \leq  i,j \leq d} : a_{ij} = a_{ji} \; \mbox{for} \; i,j=1,\cdots,d \mbox{ and } \sum_{i,j=1}^d a_{ij} x_i x_j > 0 \; \forall x_1,\cdots,  \forall x_d  \}$.  Later we will need the following theorem which gives an upper bound on the number of connected components of a semialgebraic set.

\begin{theorem}[Corollary 3.6, Chapter 3 of \cite{vdD}] \label{sa_thm}
If $S \subset \R^{k_1+k_2}$ is semialgebraic then there is a natural number $M$ such that for each point $a \in \R^{k_1}$ the fiber $S_a := \{ b \in \R^{k_2} : (a,b) \in S \}$ has at most $M$ connected components.
\end{theorem}

\subsection{Proof of Proposition \ref{bdd_comp_prop}}

The coordinate mapping $ \Psi^{-1} : X \rightarrow \R^d$ is a Hilbert space isomorphism when $\R^d$ is equipped with the  standard Euclidean inner product $\langle \cdot, \cdot \rangle$. Thus $\cC_{(X,T_\delta), \delta_0, R}(\Omega) = \cC_{(\R^d, \widetilde{T}_\delta), \delta_0, R}(\Psi^{-1}(\Omega))$, where $\widetilde{T}_\delta := \Psi^{-1} T_\delta \Psi$. Therefore, we may reduce to the case where $(X,\langle \cdot, \cdot \rangle_X) = (\R^d, \langle \cdot, \cdot \rangle)$ and the transformation $T_\delta$ on $\R^d$ is represented in Euclidean coordinates by the diagonal matrix $D_\delta= \mbox{diag}(\delta^{-m_1},\cdots,\delta^{-m_d})$ (i.e., $T_\delta(x) = D_\delta \cdot x$).

We give a proof by contradiction. Fix a one-parameter family of linear transformations $T_\delta : \R^d \rightarrow \R^d$ of the above form, and fix a closed, symmetric, convex set $\Omega \subset \R^d$, $\delta_0 > 0$, and $R \geq 1$, such that $\cC_{(\R^d, T_\delta)_{\delta > 0}, \delta_0, R}(\Omega ) \geq K_0+1$ -- we will determine the value of $K_0$ later in the argument. The family $(T_\delta)_{\delta > 0}$ satisfies the semigroup properties $T_1 = id$ and $T_{\delta_1  \delta_2} = T_{\delta_1} \circ T_{\delta_2}$. Hence, by exchanging $\Omega$ and $T_{\delta_0}(\Omega)$, we may reduce to the case $\delta_0 = 1$. The inequality $\cC_{(\R^d, T_\delta)_{\delta > 0},1, R}(\Omega ) \geq K_0+1$ implies that there exist intervals $I_1 > \cdots > I_{K_0+1}$ in $(0,1]$ and subspaces $V_1,\cdots,V_{K_0+1} \subset \R^d$ such that (a) $T_{r(I_k)} (\Omega)$ is $R$-transverse to $V_k$, whereas (b) $T_{l(I_k)} (\Omega)$ is not $256dR$-transverse to $V_k$, for all $1 \leq k \leq K_0+1$. 

The Grassmanian $\cG$ of subspaces of $\R^d$ will be endowed with the metric 
\[
d_{\cG}(V_1,V_2) := \inf \{\| U - id \|_{\mbox{\scriptsize op}} : U \in O(d,\R), \; U (V_1) = V_2 \}. 
\]
In particular, $d_{\cG}(V_1,V_2) < \infty \iff \dim(V_1) = \dim(V_2)$. 

Fix $\epsilon := (2^{12} d R^2)^{-1}$ and let $\cN$ be an $\epsilon$-net in $\cG$.

By a perturbation argument we approximate $\Omega$ by an ellipsoid $\cE$ with similar properties. Let $R_0 := 256 dR$. Fix a compact, symmetric, convex set $\widetilde{\Omega} \subset \R^d$ with nonempty interior such that
\[
\left\{
\begin{aligned}
&d_H(T_{r(I_k)}(\Omega) \cap R_0 \cB, T_{r(I_k)}(\widetilde{\Omega})\cap R_0 \cB) < R_0^{-1} , \mbox{ and} \\
&d_H(T_{l(I_k)}(\Omega) \cap R_0 \cB, T_{l(I_k)}(\widetilde{\Omega}) \cap R_0 \cB) <  R_0^{-1} \;\; \mbox{for all} \;\; 1 \leq k \leq K_0+1,
\end{aligned}
\right.
\]
where $d_H$ is the Hausdorff metric with respect to the Euclidean norm on $\R^d$ -- we can choose $\widetilde{\Omega}$ of the form $(\Omega + \lambda \cB) \cap (\lambda^{-1} \cB)$ for a small enough constant $\lambda > 0$. By Lemma \ref{comp_lem2} and properties (a) and (b), we have that $T_{r(I_k)}( \widetilde{\Omega})$ is $4R$-transverse to $V_k$, but $T_{l(I_k)}(\widetilde{\Omega})$ is not $64dR$-transverse to $V_k$.  If $\cE$ is the John ellipsoid of $\widetilde{\Omega}$, which satisfies $\cE \subset \widetilde{\Omega} \subset \sqrt{d} \cE$, then $T_{r(I_k)}(\cE)$ is $4\sqrt{d} R$-transverse to $V_k$, but $T_{l(I_k)}(\cE)$ is not $64 \sqrt{d} R$-transverse to $V_k$. Hence, setting $\widehat{R} = 16\sqrt{d} R$, 
\begin{equation}
\label{main1}
\left\{
\begin{aligned}
&T_{r(I_k)}(\cE) \mbox{ is } (1/4) \widehat{R}\mbox{-transverse to } V_k \mbox{, but } \\
&T_{l(I_k)}(\cE) \mbox{ is not } 4\widehat{R}\mbox{-transverse to } V_k, \mbox{ for all } 1 \leq k \leq K_0 + 1.
\end{aligned}
\right.
\end{equation}

We parametrize ellipsoids by positive-definite matrices in the usual way: any ellipsoid has the form $\cE_{A} := \{ x  \in \R^d : \langle A  x , x \rangle \leq 1\}$ for some $A \in \cM^+$. Furthermore, any subspace of $\R^d$ has the form $V_C := \mbox{rowsp}(C)$, where $\mbox{rowsp}(C)$ denotes the span of the row vectors of a matrix $C \in \R^{d \times d}$. Consider the set
\[
S = \{ (C,A,\overline{R}, \delta)  \in \R^{d^2}  \times \cM^+ \times  [1,\infty) \times (0,\infty): T_{\delta}(\cE_{A}) \mbox{ is } \overline{R}\mbox{-transverse to } V_C \}.
\]
Here, it is useful to note that $T_\delta(\cE_A) = \cE_{A_\delta}$, with $A_\delta := D_{\delta^{-1}} A D_{\delta^{-1}}$. Then $S$ is a semialgebraic subset of $\R^{2d^2 + 2}$ because $\cM^+$ is semialgebraic and the statement ``$T_{\delta}(\cE_{A})$ is $\overline{R}$-transverse to $V_C$'' is expressable by a first order formula in the variables $(C,A,\delta,\overline{R}) \in \R^{2d^2+2}$.

Consider the ellipsoid $\cE$ determined to satisfy \eqref{main1}, and fix an arbitrary subspace $V \subset \R^d$. Write $V=V_C$ and $\cE = \cE_A$ for some $C \in \R^{d^2}$, $A \in \cM^+$. By Theorem \ref{sa_thm}, for any $\overline{R}  \geq 1$ there exists a finite set $\Lambda = \Lambda(V_C,\cE_A,\overline{R})  \subset (0,\infty)$ with $\#(\Lambda) \leq M$, where $M$ is an integer constant determined by $d$ and $\vec{m}$, so that for any interval $I \subset (0,\infty) \setminus \Lambda$, either $(C,A,\delta,\overline{R}) \in S$ (i.e., $T_{\delta}(\cE_A)$ is $\overline{R}$-transverse to $V_C$) for all $\delta \in I$, or $(C,A,\delta,\overline{R}) \notin S$ (i.e., $T_{\delta}(\cE_A)$ is not $\overline{R}$-transverse to $V_C$) for all $\delta \in I$. Set
\[
\Lambda_{\bad} := \bigcup_{V \in \cN} \Lambda(V,\cE,\widehat{R}).
\]
For an interval $I \subset (0,\infty) \setminus \Lambda_{\bad}$ and subspace $V \in \cN$, we have (A) either  [$T_{\delta}(\cE)$ is $\widehat{R}$-transverse to $V$ for all $\delta \in I$] or [$T_{\delta}(\cE)$ is not $\widehat{R}$-transverse to $V$ for all $\delta \in I$]. Note that $\#(\Lambda_{\bad}) \leq \#(\cN) \cdot M$.

Set  $K_0 := 2 \cdot  \#(\cN) \cdot M$. Then $K_0 + 1 > 2 \cdot \#(\Lambda_{\bad})$. By definition of the order relation on intervals, at most two of the intervals $I_1 > \cdots > I_{K_0+1}$ can contain a given number $\delta \in \R$. Thus, we can find $k_*$ so that $I_{k_*}$ is disjoint from $\Lambda_{\bad}$. 

Since $\cN$ is an $\epsilon$-net in $\cG$, there exist $U \in O(d, \R)$ and $V \in \cN$ with $U(V_{k_*}) = V$ and $\| U^{-1} - id \|_{\mbox{\scriptsize op}} = \| U - id \|_{\mbox{\scriptsize op}} < \epsilon = \frac{1}{2^{12} d R^2} = \frac{1}{16 \widehat{R}^2}$. By condition (A), either $T_{\delta}(\cE)$ is $\widehat{R}$-transverse to $V$ for all $\delta \in I_{k_*}$, or $T_{\delta}(\cE)$ is not $\widehat{R}$-transverse to $V$ for all $\delta \in I_{k_*}$. By Lemma \ref{lem01},  either $T_{\delta}(\cE)$ is $(\frac{1}{4}) \widehat{R}$-transverse to $V_{k_*}$ for all $\delta \in I_{k_*}$, or $T_{\delta}(\cE)$ is not $4 \widehat{R}$-transverse to $V_{k_*}$ for all $\delta \in I_{k_*}$. This contradicts \eqref{main1} for $k=k_*$ and finishes the proof of the proposition.

\section{The Local Main Lemma}

\begin{definition}
For $x \in \R^n$, let $\cP_x = \cP$ be the Hilbert space endowed with the inner product $\langle \cdot , \cdot \rangle_{x} := \langle \cdot, \cdot \rangle_{x,1}$. Write $\cX_x$ for the system $(\cP_x, \tau_{x,\delta})_{\delta > 0}$, where the rescaling transformations $\tau_{x,\delta} : \cP_x \rightarrow \cP_x$ ($\delta > 0$) are given by $\tau_{x,\delta}(P)(z) = \delta^{-m} P(x + \delta(z-x))$. With respect to the monomial basis $\{(z-x)^\alpha\}_{|\alpha| \leq m-1}$, the transformation  $\tau_{x,\delta}$ is represented by a diagonal matrix with negative integer powers of $\delta$ on the main diagonal. Given a ball $B \subset \R^n$ and a finite set $E \subset \R^n$, the local complexity of $E$ on $B$ is  the integer-valued quantity
\[
\cC(E| B) := \sup_{x \in B} \cC_{\cX_x, C_* \diam(B),R_\lbl} (\sigma(x)).
\]
\end{definition}

\begin{rem} \label{comp_reform}
We obtain an equivalent formulation of local complexity by inspection of Definition \ref{comp_defn}: We have $\cC(E|B) \geq K$ if and only if there exists $x \in B$ and there exist subspaces $V_1,\cdots,V_K \subset \cP$ and intervals $I_1 > I_2 > \cdots > I_K$ in $(0,  \diam(B)]$, such that $\tau_{x,r(I_k)} (\sigma(x))$ is $(x,C_*, R_\lbl)$-transverse to $V_k$, but $\tau_{x,l(I_k)}( \sigma(x))$ is not $(x,C_*, R_\med)$-transverse to $V_k$ for all $k=1,\cdots,K$. Here,  $R_\med := 256 D R_\lbl$ (see \eqref{fix_c}).
\end{rem}

We have the following basic monotonicity property of complexity: $B_1 \subset B_2 \implies \cC(E|B_1) \leq \cC(E|B_2)$. As a consequence of Proposition \ref{bdd_comp_prop}, we also have the following result:

\begin{corollary}\label{bdd_comp_cor}
There exists $K_0 = K_0(m,n)$ such that $\cC(E|B) \leq K_0$ for any closed ball $B \subset \R^n$ and finite subset $E \subset \R^n$.
\end{corollary}

Next we define the (global) complexity $\cC(E)$ of a finite subset $E \subset \R^n$.

\begin{definition}\label{comp_def}
Given a finite subset $E \subset \R^n$, let $B_0 \subset \R^n$ be a compact ball containing $E$ -- for definiteness, one can choose $B_0$ to be the compact ball of minimal radius containing $E$. Then let $\cC(E) := \cC(E|5B_0)$.
\end{definition}

Now Lemma \ref{c_bd_lem} from the introduction follows from Corollary \ref{bdd_comp_cor}. The main apparatus that will be used to prove Theorem \ref{main_thm2} is the following:

\begin{lemma}[Local Main Lemma for $K$] \label{lml}

Let $K \geq -1$. There exist constants $C^\# = C^\#(K) \geq 1$ and $\ell^\# = \ell^\#(K) \in \Z_{\geq 0}$, depending only on $K,m,n$, with the following properties. 

Let $E \subset \R^n$ be finite and let  $B_0 \subset \R^n$ be a ball. If $\cC(E|5B_0) \leq K$ then the following holds:

\underline{Local Finiteness Principle on $B_0$}: Suppose $f : E \rightarrow \R$, $M > 0$, $x_0 \in B_0$, and $P_0 \in \cP$ satisfy the following \emph{finiteness hypothesis}: For all $S \subset E$ with $\#(S) \leq (D+1)^{\ell^\#}$ there exists $F^S \in C^{m-1,1}(\R^n)$ with $F^S = f$ on $S$, $J_{x_0} F^S = P_0$, and $\|F^S \| \leq M$.  Then there exists a function $F \in C^{m-1,1}(\R^n)$ with $F = f$ on $E \cap B_0$, $J_{x_0} F = P_0$, and $\|F \| \leq C^\# M$. 

\end{lemma}
\begin{rem} \label{reform_lfip}
Equivalently, the Local Finiteness Principle on $B_0$ states that
\[
\Gamma_{\ell^\#}(x_0,f,M) \subset  \Gamma_{E \cap B_0} (x_0, f, C^\# M).
\]
In particular, by taking $f = 0$ and $M=1$, we have
\[
\sigma_{\ell^\#}(x_0) \subset C^\# \cdot \sigma(x_0,E \cap B_0).
\]
\end{rem}

\subsection{Proof of Theorem \ref{main_thm2}}

We now explain why it is that the Local Main Lemma implies Theorem \ref{main_thm2}. Fix a ball $B_0$ with $E \subset B_0$ as in Definition \ref{comp_def}. We apply the Local Main Lemma for $K = \cC(E) = \cC(E|5B_0)$ and deduce that the Local Finiteness Principle for $B_0$ is true. Therefore, $\Gamma_{\ell^\#}(x_0,f,M) \subset  \Gamma_{E} (x_0, f, C^\# M)$ for any $M>0$. Our main result, Theorem \ref{main_thm2}, now follows easily: By Lemma \ref{gamma_trans_lem}, the Finiteness Hypothesis $\cF\cH(k^\#)$  (see \eqref{FH}) with constant $k^\# = (D+1)^{\ell^\# + 1}$ implies $\Gamma_{\ell^\#}(x_0,f,1) \neq \emptyset$, and so $\Gamma_{E} (x_0, f, C^\#) \neq \emptyset$. In particular, there exists $F \in C^{m-1,1}(\R^n)$ with $F = f$ on $E$ and $\| F \| \leq C^\#$. 

\begin{rem}\label{comment1}
In section \ref{dep_sec} we verify that the constant $C^\# = C^\#(K)$ in the Local Main Lemma depends exponentially on $K$, and the constant $\ell^\#=\ell^\#(K)$ depends linearly on $K$; thus, $k^\# = (D+1)^{\ell^\# + 1}$ will depend exponentially on $K$. This finishes the proof of Theorem \ref{main_thm2}.
\end{rem}

\subsection{Organization}

The rest of the paper is organized as follows.  In section \ref{sec_ind} we formulate the Main Induction Argument that will be used to prove the Local Main Lemma for all $K$. In section \ref{sec_mdl} we prove the Main Decomposition Lemma which will allow us to pass from a local extension problem on a ball $B_0$ to a family of easier subproblems on a collection of ``Whitney balls'' $B \subset 5B_0$; this lemma is the main component in the analysis of the induction step. In section \ref{sec_fmcj2}, we state a technical lemma that allows us to control the shape of the set $\sigma_{\ell}(x)$ at lengthscales which are much coarser than the lengthscales of the balls in the decomposition; we next apply this lemma to enforce mutual consistency for a family of jets that are associated to the local extension problems on the Whitney balls. In section \ref{sec_glue} we will construct a solution to the local extension problem on $B_0$ by gluing together the solutions to the local problems on the Whitney balls by means of a partition of unity; the consistency conditions arranged in the previous step will ensure that the individual local extensions are sufficiently compatible to ensure the desired regularity properties for the glued-together function.

\section{The Main Induction Argument I: Setup}\label{sec_ind}

We will prove the Local Main Lemma by induction on the complexity parameter $K \in \{-1,0,\cdots, K_0\}$ -- recall, $K_0$ is a finite upper bound on the local complexity of any set. When $K=-1$, the Local Main Lemma is vacuously true (say, for $C^\#(-1) = 1$, $\ell^\#(-1) = 0$) since complexity is non-negative. This establishes the base case of the induction.

For the induction step, fix $K \in \{0,1,\cdots, K_0 \}$. The induction hypothesis states that the Local Main Lemma for $K-1$ is valid. Denote the finiteness constants in the Local Main Lemma for $K-1$ by $\ell_\old := \ell^\#(K-1)$ and $C_\old := C^\#(K-1)$. Applying the Local Main Lemma to a closed ball of the form $(6/5) \cdot B$, we obtain
\begin{equation}\label{ind_hyp}
\begin{aligned}
&\mbox{If } \; x \in (6/5) \cdot B \mbox{ and } \cC(E|6B) \leq K-1, \; \mbox{then}, \; \\
& \Gamma_{\ell_\old}(x,f,M) \subset \Gamma_{E \cap \frac{6}{5}B}(x,f,C_\old M) \; \mbox{for any} \; f: E \rightarrow \R, \; M > 0.
\end{aligned}
\end{equation}
(Here we use the formulation of the Local Finiteness Principle in  Remark \ref{reform_lfip}.)

Fix a ball $B_0 \subset \R^n$ with $\cC(E|5B_0) \leq K$. To prove the Local Main Lemma for $K$, we are required to prove the Local Finiteness Principle (LFP) on $B_0$ for a suitable choice of the constants $\ell^\# \in \Z_{\geq 0}$ and $C^\# \geq 1$, determined by $m$, $n$, and $K$. Thus, our goal is to prove that $\Gamma_{\ell^\#}(x_0,f,M) \subset \Gamma_{E \cap B_0}(x_0,f,C^\# M)$ for any $f: E \rightarrow \R$, $x_0 \in B_0$, $M > 0$. A rescaling of the form $f \mapsto f/M$ allows us to reduce to the case $M=1$. 

If $\#(B_0 \cap E) \leq 1$ then the LFP is true as long as $C^\# \geq 1$ and $\ell^\# \geq 0$ -- indeed, 
\begin{equation}
\label{triv_case}
\begin{aligned}
\Gamma_{\ell^\#} (x_0, f, 1) &\subset \Gamma_0(x_0,f,1) = \bigcap_{S \subset E, \; \#(S) \leq 1} \Gamma_S(x_0,f,1)  \\
& \subset \Gamma_{E \cap B_0} (x_0,f,1) \subset \Gamma_{E \cap B_0} (x_0,f,C^\#).
\end{aligned}
\end{equation}
Accordingly, it suffices to assume that
\begin{equation}
\label{twopoints}
\#(B_0 \cap E) \geq 2.
\end{equation}
Under these assumptions, we will prove that for any $x_0 \in B_0$ and $f: E \rightarrow \R$, 
\begin{equation}
\label{main-task}
\Gamma_{\ell^\#}(x_0,f,1) \subset \Gamma_{E \cap B_0}(x_0,f,C^\#).
\end{equation}

\section{The Main Decomposition Lemma}\label{sec_mdl}

In this section we fix the following data: A closed ball $B_0 \subset \R^n$; a point $x_0 \in B_0$; an integer $K \in \Z_{\geq 0}$; a finite set $E \subset \R^n$ satisfying $\#(E \cap B_0) \geq 2$ and $\cC(E|5B_0) \leq K$; a function $f : E \rightarrow \R$; an integer $\ell^\#  \in \Z_{\geq 0}$; and a polynomial $P_0 \in \Gamma_{\ell^\#}(x_0,f,1)$.

Given the data $(B_0,x_0,K,E,f,\ell^\#,P_0)$, we will now explain how to produce a Whitney cover $\cW$ of the ball $2B_0$ so as to decompose the local extension problem on $B_0$ into a collection of easier local extension problems associated to the balls $B \in \cW$.

\begin{lemma}[Main Decomposition Lemma]
\label{mdl}
Define the constants $R_\lbl \ll R_\med \ll R_\big \ll R_\huge$, $C_*$, and $C_{**}$ as in \eqref{fix_c}. Fix data $(B_0,x_0,K,E,f,\ell^\#,P_0)$ as above. Then there exists a subspace $V \in \dti$ such that

(a) The set $\sigma(x)$ is $(x, C_* \diam(B_0), R_\lbl)$-transverse to $V$ for all $x \in 100 B_0$.

\noindent There exists a Whitney cover $\cW$ of $2B_0$ such that, for all $B \in \cW$,

(b) $B \subset 100 B_0$ and $\diam(B) \leq \frac{1}{2} \diam(B_0)$.

(c) The set $\sigma(x)$ is $(x, C_* \delta, R_\huge)$-transverse to $V$ for all $x \in 8B$, $\delta \in [\diam(B),\diam(B_0)]$. 

(d) Either $\#(6B \cap E) \leq 1$ or $\cC(E|6B) < K$.

\noindent For every $B \in \cW$ there exists a point $z_B \in \R^n$ and a jet $P_B \in \cP$ satisfying

(e)  $z_B \in \frac{6}{5}B \cap 2B_0$; also, if $x_0 \in \frac{6}{5}B$ then $z_B = x_0$.

(f) $P_B\in \Gamma_{\ell^\#-1}(z_B,f,C_{**})$ and $P_0 - P_B \in C_{**} \cB_{z_B,\diam(B_0)}$; also, if $x_0 \in \frac{6}{5}B$ then $P_B = P_0$.

(g) $P_0 - P_B \in V$.
\end{lemma}

Let $\cW$ be the Whitney cover arising from Lemma \ref{mdl}. We obtain a local finiteness principle for the balls $B \in \cW$ in the next lemma.
\begin{lemma}
\label{old_fp_lem}
The Local Finiteness Principle on $\frac{6}{5}B$ is true for every $B \in \cW$ with finiteness constants $\ell_\old = \ell^\#(K-1) \in \Z_{\geq 0}$ and $C_\old = C^\#(K-1) \geq 1$. That is, 
\[
\Gamma_{\ell_\old}(x,f,M) \subset \Gamma_{E \cap \frac{6}{5}B}(x,f,C_\old M)\mbox{, for all } B \in \cW, \; x \in \frac{6}{5}B, \; M > 0.
\]
\end{lemma}
\begin{proof}
If $\cC(E | 6B) < K$, the conclusion follows from \eqref{ind_hyp}. On the other hand, if $\#(E \cap 6B) \leq 1$, the result follows from \eqref{triv_case}. Condition (d) states that these two cases are exhaustive. \end{proof}

\subsection{Proof of the Main Decomposition Lemma}\label{proof_mdl}

We first prove property (a). By Lemma \ref{main-labeling-lemma} there exists a subspace $V \in \dti$ such that
 \[
 \sigma(x) \mbox{ is } (x, C_* \diam(B_0), R_\lbl)\mbox{-transverse to } V \mbox{ for all }x \in 100B_0.
 \] 
 This establishes property (a). The construction of $\cW$ is based on the following definition:

\begin{definition}
A ball $B \subset 100 B_0$ is OK if $\#(B \cap E) \geq 2$ and if there exists $z \in B$ such that $\sigma(z)$ is $(z, C_* \delta, R_\big)$-transverse to $V$ for all $\delta \in [\diam(B), \diam(B_0)]$.
\end{definition}

The OK property is \emph{inclusion monotone} in the sense that if $B \subset B' \subset 100 B_0$ and $B$ is OK then $B'$ is OK.

For each $x \in 2B_0$, let $r(x) := \inf \{r > 0:  B(x,r)\subset 100 B_0, \; B(x,r) \mbox{ is OK}\}$. Here, we write $B(x,r)$ to denote the closed Euclidean ball with center $x$ and radius $r$. Every ball that is contained in $100B_0$ and that contains $2B_0$ is OK, so the previous infimum is well-defined -- this also implies that $r(x) \leq 2 \diam(B_0)$ for all $x \in 2B_0$. If the radius of $B \subset 100 B_0$ is less than the quantity $\Delta := \frac{1}{2} \min \{ |x-y| : x,y \in E, x \neq y \} > 0$ then $\#(B \cap E) \leq 1$, and therefore $B$ is not OK -- in particular, this shows that $r(x) \geq \Delta$ for all $x \in 2B_0$. Let $B_x := B(x, r(x)/7)$ for $x \in 2B_0$. Then
\begin{equation}\label{inc1}
70B_x \subset 100 B_0, \quad \mbox{for } x \in 2B_0.
\end{equation}
Obviously the family $\cW^* = \{B_x \}_{x \in 2B_0}$ is a cover of $2B_0$.
\begin{lemma}\label{ok_lem}
If $B \in \cW^*$ then $8B$ is OK, and $6B$ is not OK.
\end{lemma}
\begin{proof}
We write $B = B(x, r(x)/7)$ for $x \in 2B_0$. According to \eqref{inc1}, $6B \subset 8B \subset 100B_0$. By definition of $r(x)$ as an infimum and the inclusion monotonicity of the OK property, the result follows.
\end{proof}
We apply the Vitali covering lemma to extract a finite subcover $\cW \subset \cW^*$ of $2 B_0$ with the property that the family of third-dilates $\{\frac{1}{3} B\}_{B \in \cW}$ is pairwise disjoint.

\begin{lemma}
$\cW$ is a Whitney cover of $2 B_0$.
\end{lemma}
\begin{proof}
We have only to verify condition (c) in the definition of a Whitney cover (see Definition \ref{Whit_cov_def}). Suppose for sake of contradiction that there exist balls $B_j = B(x_j,r_j) \in \cW$ for $j=1,2$, with $\frac{6}{5} B_1 \cap \frac{6}{5} B_2 \neq \emptyset$ and $r_1 < \frac{1}{8} r_2$. Since $\frac{6}{5} B_1 \cap \frac{6}{5} B_2 \neq \emptyset$, we have $|x_1-x_2| \leq \frac{6}{5}r_1+ \frac{6}{5} r_2$. If  $z \in 8 B_1$ then $|z-x_1| \leq 8 r_1$, and therefore
\[
|z-x_2| \leq |z-x_1| + |x_1-x_2| \leq 8 r_1 + \frac{6}{5} r_1 + \frac{6}{5} r_2 \leq r_2 + \frac{3}{20} r_2 + \frac{6}{5} r_2 \leq 6 r_2.
\] 
Hence, $8 B_1 \subset 6 B_2$. By Lemma \ref{ok_lem}, $8 B_1$ is OK. Thus, by inclusion monotonicity, $6B_2$ is OK. But this contradicts Lemma \ref{ok_lem}. This finishes the proof by contradiction.
\end{proof}

We now establish conditions (b)-(d) in the Main Decomposition Lemma. 

Fix a ball $B \in \cW$. 

We will use the following principal condition: (PC) If  $\#(6 B \cap E) \geq 2$ then for all $x \in 6B$ there exists $\delta_x \in [ 6\diam(B), \diam(B_0)]$ so that $\sigma(x)$ is not $(x, C_* \delta_{x}, R_\big)$-transverse to $V$. This condition follows because $6B$ is not OK.

\noindent \emph{Proof of (b):} The inclusion  $B \subset 100B_0$ follows from \eqref{inc1}. For sake of contradiction, suppose that $ \diam(B) > \frac{1}{2} \diam(B_0)$. Since $B \cap B_0 \neq \emptyset$, we have $B_0 \subset 5B$. Therefore, $\#(5B \cap E) \geq \#(B_0 \cap E) \geq 2$. Fix a point $x \in B$. Then (PC) implies that the interval $[6 \diam(B),\diam(B_0)]$ is nonempty, thus $\diam(B) \leq \frac{1}{6} \diam(B_0)$, which gives the contradiction.

\noindent  \emph{Proof of (c):} Since $8B$ is OK, $\sigma(z)$ is $(z, C_* \delta, R_\big)$-transverse to $V$ for some $z \in 8B$ and all $\delta \in [ 8\diam(B),\diam(B_0)]$. If $x \in 8B$ then $|x-z| \leq 8 \diam(B) \leq \delta \leq \frac{c_1}{R_\big} \cdot (C_* \delta)$ (see \eqref{fix_c}), and so, by Lemma \ref{trans_lem},
\[
\sigma(x) \mbox{ is } (x, C_* \delta , 8 R_\big)\mbox{-transverse to } V  \mbox{ if } \delta \in [8\diam(B),\diam(B_0)].
\]
Any number in the interval $[\diam(B),\diam(B_0)]$ is comparable to a number in $[8\diam(B),\diam(B_0)]$ up to a multiplicative  factor of at most $8$. Hence, by Lemma \ref{lem00}, $\sigma(x)$ is $(x, C_* \delta ,  8^{m+1} R_\big)$-transverse to $V$ for all $\delta \in [\diam(B),\diam(B_0)]$. Since $R_\huge \geq 8^{m+1} R_\big$ (see \eqref{fix_c}), this implies (c).

\noindent  \emph{Proof of (d):}  Suppose that $\#(6B \cap E) \geq 2$ and set $J := \cC(E|6B)$. According to the formulation of complexity in Remark \ref{comp_reform}, there exist intervals $I_1 > I_2 > \cdots > I_J$ in $(0, 6  \diam(B)]$, subspaces $V_1, \cdots, V_J \subset \cP$, and a point $z \in 6B$, such that\\
(A) $\tau_{z,r(I_j)} ( \sigma(z))$ is $(z, C_*, R_\lbl)$-transverse to $V_j$, and\\
(B) $\tau_{z,l(I_j)} ( \sigma(z))$ is not $(z, C_* , R_\med)$-transverse to $V_j$, for $1 \leq j \leq J$, where $R_\med = 256DR_\lbl$.

Since $B \cap B_0 \neq \emptyset$ and $\diam(B) \leq \frac{1}{2} \diam(B_0)$ (see (b)) it follows that $6B \subset 5 B_0$. Hence, $z \in 5B_0$. 

Since $\#(6B \cap E) \geq 2$, (PC) implies that there exists $\delta_z \in [6 \diam(B), \diam(B_0)]$ with
\begin{equation}
\label{eqn12}
\sigma(z) \mbox{ is not } (z, C_*\delta_z,R_\big)\mbox{-transverse to } V.
\end{equation}
We will now establish that (A) and (B) hold for $j=0$ with $I_0 := [ \delta_z,  \diam(B_0)]$ and $V_0 := V$. Since $V$ is a DTI subspace, $\tau_{z,l(I_0)} V = V$, and therefore, by rescaling \eqref{eqn12},
\begin{equation}
\label{a1}
\tau_{z,l(I_0)}(\sigma(z)) \mbox { is not } (z,C_*,R_\big)\mbox{-transverse to } V.
\end{equation}
On the other hand, from property (a) we learn that $\sigma(z)$ is $(z, C_* \diam(B_0), R_\lbl)$-transverse to $V$. Therefore, by rescaling,
\begin{equation}
\label{a2}
\tau_{z,r(I_0)}(\sigma(z)) \mbox { is } (z,C_*,R_\lbl)\mbox{-transverse to } V.
\end{equation}
The conditions \eqref{a1} and \eqref{a2} together imply (A) and (B) for $j=0$ (recall $R_\big \geq R_\med$).

Notice that $r(I_1) \leq  6 \diam(B) \leq  \delta_z =  l(I_0)$, thus $I_1 < I_0$. In conclusion,  $I_0 > I_1 > \cdots > I_J$ are subintervals of $(0, \diam(B_0)]$.

We produced intervals $I_0 > I_1 > \cdots > I_J$ in $(0, 5 \diam(B_0)]$ and subspaces $V_0,\cdots,V_J \subset \cP$, so that (A) and (B) hold for $j = 0,1,\cdots,J$. Since $z \in 5 B_0$, by the formulation of complexity in Remark \ref{comp_reform}, we have $\cC(E|5B_0) \geq J+1$. Since $\cC(E|5B_0) \leq K$, this completes the proof of (d).

Finally we define a collection of points $\{z_B\}_{B \in \cW}$ and polynomials $\{P_B\}_{B \in \cW}$ so as to establish properties (e)-(g).

\noindent \emph{Proof of (e):} We define the collection $\{z_B\}_{B \in \cW}$ to satisfy property (e). For all $B \in \cW$ such that $x_0 \in \frac{6}{5}B$ we set $P_B= P_0$. We define $P_B$ for the remaining balls $B \in \cW$ in the proof of (f) and (g) below.

\noindent  \emph{Proofs of (f) and (g):} If $x_0 \in \frac{6}{5}B$ then $z_B = x_0$ and $P_B = P_0$, in which case (f) and (g) are trivially true (note that $P_0 \in \Gamma_{\ell^\#}(x_0,f,1) \subset \Gamma_{\ell^\#-1}(x_0,f,1)$). Suppose instead $x_0 \notin \frac{6}{5}B$. Then $z_B \in \frac{6}{5}B \cap 2B_0$ and so $|x_0-z_B| \leq 2\diam(B_0)$. By Lemma \ref{gamma_trans_lem}, given that $P_0 \in \Gamma_{\ell^\#}(x_0,f,1)$, we can find $P_B \in \Gamma_{\ell^\# - 1}(z_B,f,1)$ with $ P_0 - P_B  \in C_T  \cB_{z_B, 2\diam(B_0)} \subset 2^m C_T \cB_{z_B,\diam(B_0)}$. We still have to arrange $P_0 - P_B \in V$ as in (g). Unfortunately, there is no reason for this to be true, and we will have to perturb $P_B$ to arrange this property. This is where we use condition (a), which implies that $\sigma(z_B)$ is $(z_B,5C_* \diam(B_0),R_\lbl)$-transverse to $V$. Therefore,
\[
\begin{aligned}
\cB_{z_B,  \diam(B_0)}/V \subset \cB_{z_B, 5C_* \diam(B_0)}/ V &\subset R_\lbl \cdot (\sigma(z_B) \cap \cB_{z_B, 5C_* \diam(B_0)})/V\\
&\subset R_\lbl \cdot (\sigma_{\ell^\#-1}(z_B) \cap \cB_{z_B, 5C_* \diam(B_0)})/V.
\end{aligned}
\]
Since $P_0 - P_B \in 2^m C_T \cB_{z_B, \diam(B_0)}$, the last inclusion implies we can find a bounded correction
\[
R_B \in 2^m C_T R_\lbl \cdot(\sigma_{\ell^\#-1}(z_B) \cap \cB_{z_B, 5C_* \diam(B_0)}),
\]
so that 
\[
R_B/V = (P_0-P_B)/V.
\]
That is,
\[
P_B + R_B - P_0 \in V.
\] 
Set $\widetilde{P}_B = P_B + R_B$. Then $\widetilde{P}_B - P_0 \in V$ and
\[
\widetilde{P}_B \in \Gamma_{\ell^\# - 1}(z_B,f,1) + 2^m C_T R_\lbl \sigma_{\ell^\# - 1}(z_B) \subset \Gamma_{\ell^\#-1}(z_B,f, 1 + 2^m C_T R_\lbl).
\]
Furthermore, 
\[
\begin{aligned}
P_0 - \widetilde{P}_B = (P_0 - P_B) - R_B & \in 2^m C_T \cB_{z_B, \diam(B_0)} + 2^m C_T R_\lbl \cB_{z_B, 5 C_* \diam(B_0)} \\
&\subset 2^m C_T  \cdot (1 +  R_\lbl \cdot (5C_*)^m) \cdot \cB_{z_B,\diam(B_0)}.
\end{aligned}
\]
Thus we have proven (f) and (g) for all $B \in \cW$ such that $x_0 \notin \frac{6}{5}B$, with $\widetilde{P}_B$ in place of $P_B$, and with $C_{**} =1 + 2^m C_T  \cdot (1 +  R_\lbl \cdot (5C_*)^m)$. This finishes the proof of Lemma \ref{mdl}.

\section{The Main Induction Argument II}\label{sec_fmcj2}

We return to the proof of the Main Induction Argument as laid out in section \ref{sec_ind}. Fix data $(B_0,x_0,E,f)$ as in section \ref{sec_ind}; thus, $\cC(E | 5 B_0) \leq K$, $\#(E \cap B_0) \geq 2$, $x_0 \in B_0$, and $f : E \rightarrow \R$. Our goal is to establish the inclusion \eqref{main-task} for a suitable choice of the parameters $\ell^\# = \ell^\#(K)$ and $C^\#=C^\#(K)$. Our argument will require conditions on $C^\#$ and $\ell^\#$ of the form
\begin{align}\label{eqn:consts1}
&\ell^\# \geq \ell_{\old} + \overline{\chi}, \\
\label{eqn:consts2}
&C^\# \geq \widehat{C} \cdot C_{\old},
\end{align}
where $\widehat{C} \geq 1$, $\overline{\chi} \geq 1$ are determined by $m$ and $n$, and with $\ell_{\old} = \ell^\#(K-1)$, $C_\old = C^\#(K-1)$ as in \eqref{ind_hyp}. Only at the end of the argument will we fix a choice of $\ell^\#$ and $C^\#$ as above.

We fix a polynomial $P_0 \in \Gamma_{\ell^\#}(x_0,f,1)$, and apply Lemma \ref{mdl} to the data $(B_0,x_0,K,E,f,\ell^\#,P_0)$. Through this we obtain a Whitney cover $\cW$ of $2B_0$, a DTI subspace $V \subset \cP$, and the families $\{P_B\}_{B \in \cW} \subset \cP$ and $\{z_B\}_{B \in \cW} \subset \R^n$; these new data satisfy conditions (a)--(g) of Lemma \ref{mdl}.

We introduce a Whitney cover $\cW_0$ of $B_0$ by setting
\begin{equation}\label{defn:W0}
\cW_0 := \{ B \in \cW : B \cap B_0 \neq \emptyset\}.
\end{equation}

Fix constants $A \geq 10$ and $\epsilon_A \in (0,1/300]$, determined by $m$ and $n$, and defined as follows:
\begin{equation}\label{defn:A}
\begin{aligned}
&A = 2 C_{\ref{lemma2}} \cdot C_*^m \cdot R_{\huge} \\
&\epsilon_A = 1/3A^2.
\end{aligned}
\end{equation}
Here, $C_{\ref{lemma2}}$ is the constant $C$ arising in Lemma \ref{lemma2}.

Our next result states that the polynomials $\{ P_B\}_{B \in \cW_0}$ are pairwise compatible.

\begin{lemma}\label{FMCJ_lem}
There exist constants $\overline{\chi} \geq 1$ and $\overline{C} \geq 1$, determined by $m$ and $n$, such that the following holds. Suppose that $\ell^\# \in \N$ and the family $(P_B)_{B \in \cW}$ are as in the statement of Lemma \ref{mdl}, and suppose that  $\ell^\# \geq \ell_{\old} + \overline{\chi}$. Then $P_B - P_{B'} \in \overline{C} \cdot C_{\old} \cdot \cB_{z_B,\diam(B)}$ for any $B,B' \in \cW_0$ with $(\frac{6}{5})B \cap (\frac{6}{5})B' \neq \emptyset$.
\end{lemma}

We will see that Lemma \ref{FMCJ_lem} follows easily from the next result.

\begin{lemma}\label{main_lem}
There exist constants $\chi \geq 1$, and $C \geq 1$, depending only on $m$ and $n$, such that the following holds. Suppose there exists a ball $\widehat{B} \in \cW_0$ satisfying $\diam(\widehat{B}) \leq \epsilon_A  \cdot \diam(B_0)$. Then 
\[
(\sigma_{\ell+1}(x) + \cB_{z_B, \diam(B)}) \cap V \subset C C_{\old} \cB_{z_B,\diam(B)}
\]
for any $B \in \cW_0$, $x \in 3B$, and $\ell \geq \ell_{\old} + \chi$.
\end{lemma}

Lemma \ref{main_lem} is difficult for subtle reasons: We know from condition (c) of the Main Decomposition Lemma that $\sigma(x)$ is $(x,\diam(B),R)$-transverse to $V$ for any $B \in \cW$ and $x \in 8B$, where $R = R_\huge \cdot (6 C_*)^m$. But it is not apparent why $V$ would also be transverse to $\sigma_\ell(x)$, which generally can be significantly larger than $\sigma(x)$. The key point in the proof of this proposition is that we are able to use the validity of the Local Finiteness Principle on the balls $B$ in $\cW$ to establish a two-sided relationship between the sets $\sigma(x)$ and $\sigma_{\ell^*}(x)$ (for sufficiently large $\ell^*$) as long as we are willing to ``blur'' these sets at a lengthscale larger than $\diam(B)$. Since transversality is stable under ``blurrings'' (e.g., see Lemma \ref{lem000}), the result will follow.

The proof of Lemma \ref{main_lem} is the most technical part of the paper. We next explain how Lemma \ref{main_lem} can be used to prove Lemma \ref{FMCJ_lem}. After this we will establish a preparatory result, Lemma \ref{fip_lem}. We finally give the proof of Lemma \ref{main_lem} in section \ref{proof_main_lem}.

\begin{proof}[Proof of Lemma \ref{FMCJ_lem}]
We fix $\chi$ and $C$ in Lemma \ref{main_lem}, and define $\overline{\chi} = \chi + 3$. We suppose $\ell^\# \in \N$ is picked so that $\ell^\# \geq \ell_{\old} + \overline{\chi}$. We fix $B,B' \in \cW_0$ with $\frac{6}{5} B \cap \frac{6}{5}B' \neq \emptyset$.

Consider the following two cases for the Whitney cover $\cW_0$  defined in \eqref{defn:W0}: 
\begin{itemize}
\item \emph{Case 1:} $\diam(B) > \epsilon_A \diam(B_0)$ for all $B \in \cW_0$.
\item \emph{Case 2:} There exists $\widehat{B} \in \cW_0$ with $\diam(\widehat{B}) \leq \epsilon_A \diam(B_0)$.
\end{itemize}

In Case 1, we apply condition (f) in Lemma \ref{mdl} to obtain
\begin{equation}\label{eqn:801}
P_B - P_{B'} = (P_B - P_0) + (P_0 - P_{B'}) \in C_{**} \cB_{z_B,\diam(B_0)} + C_{**} \cB_{z_{B'},\diam(B_0)}.
\end{equation}
Because $z_B,z_{B'} \in 2B_0$, we have $|z_B-z_{B'}| \leq 2 \diam(B_0)$. So by \eqref{trans_norm}, we have 
\begin{equation}\label{eqn:802}
\cB_{z_{B'},\diam(B_0)} \subset C 2^{m-1} \cB_{z_B,\diam(B_0)}.
\end{equation}
From $\diam(B) > \epsilon_A \diam(B_0)$ we conclude that 
\begin{equation}\label{eqn:803}
\cB_{z_B,\diam(B_0)} \subset (\epsilon_A)^{-m} \cB_{z_B,\diam(B)}.
\end{equation}
When put together, \eqref{eqn:801}, \eqref{eqn:802}, \eqref{eqn:803} give that $P_B - P_{B'} \in C_{**} \cdot (\epsilon_A)^{-m} (1+ C2^{m-1}) \cB_{z_B,\diam(B)}$. This completes the proof of Lemma \ref{FMCJ_lem} in Case 1.

Now suppose that Case 2 holds. By condition (g) in Lemma \ref{mdl}, we have
\[
P_B - P_{B'} = (P_B - P_0) + (P_0 - P_{B'}) \in V.
\]
By condition (f) in Lemma \ref{mdl}, we have $P_{B'} \in \Gamma_{\ell^\#-1}(z_{B'},f,C)$. By Lemma \ref{gamma_trans_lem}, there exists $\tilde{P}_B \in \Gamma_{\ell^\#-2}(z_{B}, f , C)$ with $\tilde{P}_B - P_{B'} \in C' \cdot \cB_{z_{B}, \diam(B)}$. Furthermore, since $\tilde{P}_B  \in \Gamma_{\ell^\#-2}(z_{B}, f , C) $ and $P_{B} \in \Gamma_{\ell^\#-1}(z_{B},f,C) \subset \Gamma_{\ell^\#-2}(z_{B},f,C)$, we have
\[
\tilde{P}_B - P_B \in 2C \cdot \sigma_{\ell^\#-2}(z_{B}).\]
Thus,
\[
\begin{aligned}
P_B - P_{B'} = (P_B - \tilde{P}_B) + (\tilde{P}_B - P_{B'}) &\in 2C \cdot \sigma_{\ell^\#-2}(z_{B}) +  C' \cdot \cB_{z_{B}, \diam(B)} \\
&\subset C'' \cdot (  \sigma_{\ell^\#-2}(z_{B}) + \cB_{z_{B}, \diam(B)} ),
\end{aligned}
\]
and hence
\[
P_B - P_{B'} \in C'' \cdot (  \sigma_{\ell^\#-2}(z_{B}) + \cB_{z_{B}, \diam(B)} ) \cap V.
\]
Note that $\ell^\# -3 \geq \ell_{\old} + \overline{\chi} - 3 = \ell_{\old} + \chi$. Thus, we can apply Lemma \ref{main_lem} to deduce that 
\[
(  \sigma_{\ell^\#-2}(z_{B}) + \cB_{z_{B}, \diam(B)} ) \cap V \subset C C_{\old} \cB_{z_B,\diam(B)}.
\]
Therefore, $P_B - P_{B'} \in C''' C_{\old} \cdot \cB_{z_B,\diam(B)}$, which concludes the proof of Lemma \ref{FMCJ_lem}.

\end{proof}

\subsection{Finiteness principles for set unions with weakly controlled constants}

Through the use of Lemma \ref{pou_lem2} and Helly's theorem we will obtain the following result: If a ball $\widehat{B} \subset \R^n$ is covered by a collection of balls each of which satisfies a Local Finiteness Principle, then $\widehat{B}$ satisfies a Local Finiteness Principle with constants that may depend on the cardinality of the cover. We should remark that we lack any control on the cardinality of the cover $\cW_0$ of $B_0$, and so this type of result cannot be used to obtain a Local Finiteness Principle on $B_0$ with any control on the constants. This lemma will be used in the next subsection, however, to obtain a local finiteness principle on a family of intermediate balls that are much larger than the balls of the cover, yet small when compared to $B_0$.

\begin{lemma}\label{fip_lem}
Fix $C_0 \geq 1$ and $\ell_0 \in \Z_{\geq 0}$. Let $\cW$ be a Whitney cover of a ball $\widehat{B} \subset \R^n$ with cardinality $N= \# \cW$. If the Local Finiteness Principle holds on $\frac{6}{5}B$ with constants $C_0$ and $\ell_0$, for all $B \in \cW$, then the Local Finiteness Principle holds on $\widehat{B}$ with constants $C_1$ and $\ell_1 :=  \ell_0 + \lceil \frac{\log(D \cdot N + 1)}{\log(D+1)} \rceil$, where $C_1 = C \cdot C_0$, and $C$ is determined only by $m$ and $n$ -- in particular, $C_1$ is independent of the cardinality $N$ of the cover.
\end{lemma}

\begin{proof}

Let $f : E \rightarrow \R$ and $M>0$. For any $B \in \cW$ and $x \in \frac{6}{5} B$ we have $\Gamma_{\ell_0}(x,f,M) \subset \Gamma_{E \cap \frac{6}{5} B}(x,f,C_0M)$ thanks to the Local Finiteness Principle on $\frac{6}{5}B$. Fix a point $x_0 \in \widehat{B}$. Our goal is to prove that
\begin{equation}
\label{local_fp1}
\Gamma_{\ell_1}(x_0,f,M) \subset \Gamma_{E \cap \widehat{B}}(x_0, f, C_1 M),
\end{equation}
for a constant  $C_1 \geq 1$, to be determined later.

For each $B \in \cW$, we fix $x_B \in \frac{6}{5}B$ so that
\begin{equation}\label{pts11}
x_B = x_{0} \iff x_{0} \in \frac{6}{5} B; 
\end{equation}
otherwise, if $x_0 \notin \frac{6}{5} B$ then $x_B$ is an arbitrary element of $\frac{6}{5}B$.

Fix an arbitrary element $P \in \Gamma_{\ell_1} (x_0, f,M)$.  We will define a family of auxiliary convex sets to which we will apply Helly's theorem and obtain the desired conclusion. The convex sets will belong to the vector space $\cP^N$ consisting of $N$-tuples of $(m-1)$-st order Taylor polynomials indexed by the elements of the cover $\cW$. For each $S \subset E$, the convex set $\cK(S, M) \subset \cP^N$ is defined by
\[
\begin{aligned}
\cK(S, M) := \{ (J_{x_B} F)_{B \in \cW} : F \in C^{m-1,1}(\R^n), \;  \| F \| \leq M, \; F = f \mbox{ on } S, \; J_{x_0} F  = P  \}
\end{aligned}
\]
If $\#(S) \leq (D+1)^{\ell_1}$ then $P \in \Gamma_{\ell_1}(x_0,f,M) \subset \Gamma_S(x_0,f,M)$. Thus, there exists $F \in C^{m-1,1}(\R^n)$ with $\| F \| \leq M$, $F = f$ on $S$, and $J_{x_0} F = P$. Therefore, $(J_{x_B} F)_{B \in \cW} \in \cK(S,M)$. In particular, $\cK(S,M) \neq \emptyset$ if $\#(S) \leq (D+1)^{\ell_1}$.

If $S_1,\cdots, S_{J} \subset E$, with $J :=  \dim(\cP^N) + 1 = D \cdot N  + 1$, then
\[
\bigcap_{j=1}^J \cK(S_j,M) \supset \cK( S, M), \mbox{ for } S = S_1 \cup \cdots \cup S_J.
\]
If furthermore $\#(S_j) \leq (D+1)^{\ell_0}$ for every $j$, then $\#(S) \leq J \cdot (D+1)^{\ell_0} \leq (D+1)^{ \ell_1}$, and consequently by the previous remark $\cK(S,M) \neq \emptyset$. Therefore, given arbitrary subsets $S_1,\cdots,S_J \subset E$ ($J = \dim(\cP^N) + 1$) with $\#(S_j) \leq (D+1)^{\ell_0}$ for each $j$, we have
\[
\bigcap_{j=1}^J \cK(S_j,M) \neq \emptyset.
\]
Therefore, by Helly's theorem,
\[
\cK :=  \bigcap_{\substack{S \subset E \\ \#(S) \leq (D+1)^{\ell_0}}} \cK(S,M) \neq \emptyset.
\]
Fix an arbitrary element $(P_B)_{B \in \cW} \in \cK$. By definition of $\cK$, \\
$(*)$ for any $S \subset E$ with $\#(S) \leq (D+1)^{\ell_0}$ there exists a function $F^S \in C^{m-1,1}(\R^n)$ with $\| F^S \| \leq M$, $F^S = f$ on $S$, $J_{x_0} F^S = P$, and $J_{x_B} F^S = P_B$ for all $B \in \cW$. From this condition we will establish the following properties:\\
(a) $P_B = P$ if $x_0 \in \frac{6}{5} B$, \\
(b) $| P_B - P_{B'} |_{x_B, \diam(B)} \leq C M$ whenever $\frac{6}{5} B \cap \frac{6}{5} B' \neq \emptyset$, \\
(c) for each $B \in \cW$ there exists $F_B \in C^{m-1,1}(\R^n)$ such that $\|F_B \| \leq C_0 M$, $F_B = f$ on $E \cap \frac{6}{5}B$, and $J_{x_B} F_B = P_B$.

For the proof of (a) and (b)  take $S = \emptyset$ in  $(*)$. Then $P_B = J_{x_B} F^\emptyset = J_{x_0} F^\emptyset = P$ whenever $x_0 \in \frac{6}{5}B$ (see \eqref{pts11}), which yields (a). For (b), note that $x_B \in \frac{6}{5} B$, $x_{B'} \in \frac{6}{5} B'$, and $\frac{6}{5} B \cap \frac{6}{5} B' \neq \emptyset$, and hence by the definition of Whitney covers, $\diam(B)$ and $\diam(B')$ differ by a factor of at most $8$. Thus, $|x_B - x_{B'} | \leq \frac{6}{5} \diam(B) + \frac{6}{5} \diam(B') \leq 11 \diam(B)$. Thus, by \eqref{eq:change_delta} and Taylor's theorem (see \eqref{taylor_thm}),
\[
\begin{aligned}
| P_B - P_{B'} |_{x_B, \diam(B) } &\leq 11^m | P_B - P_{B'} |_{x_B, 11 \diam(B) } \\
&= 11^m | J_{x_B} F^\emptyset - J_{x_{B'}} F^\emptyset |_{x_B, 11 \diam(B)} \\
&\leq 11^m C_T \| F^\emptyset \| \leq CM.
\end{aligned}
\]
Here, $C$ is determined only by $m$ and $n$.

For the proof of  (c), note that $(*)$ implies $P_B \in \Gamma_{\ell_0}(x_B, f,M)$ for each $B \in \cW$. By assumption, the Local Finiteness Principle holds on $\frac{6}{5}B$ with constants $C_0$ and $\ell_0$, and therefore $P_B \in \Gamma_{E \cap \frac{6}{5}B} (x_B, f, C_0 M)$ for each $B \in \cW$. This completes the proof of (c). 

Fix a partition of unity $\{\theta_B\}$ adapted to the Whitney cover $\cW$ as in Lemma \ref{pou_lem1}, and set $F = \sum_{B \in \cW} \theta_B F_B$. By use of properties (b) and (c), we conclude via Lemma \ref{pou_lem2} that (A) $\| F \|_{C^{m-1,1}(\widehat{B})} \leq C' \cdot C_0 \cdot M$ and (B) $F = f$ on $E \cap \widehat{B}$; here, $C'$ is determined by $m$ and $n$.  Since $\supp \theta_B \subset \frac{6}{5} B$, we learn that  $J_{x_0} \theta_B = 0$ if $x_0 \notin \frac{6}{5} B$; on the other hand,  $J_{x_0} F_B = J_{x_B} F_B = P_B = P$ if $x_0 \in \frac{6}{5} B$ (see \eqref{pts11}). Thus, if we compare the following sums term-by-term, we obtain the identity
\[
J_{x_0} F  =  \sum_{B \in \cW} J_{x_0} \theta_B \odot_{x_0} J_{x_0} F_B =  \sum_{B \in \cW} J_{x_0} \theta_B \odot_{x_0} P.
\]
Recall that $\sum_{B \in \cW} \theta_B = 1$ on $\widehat{B}$ and $x_{0} \in \widehat{B}$. Thus, $\sum_{B \in \cW} J_{x_0} \theta_B =  J_{x_0} (1) = 1$. Therefore, (C) $J_{x_0} F = P$. By a standard technique we extend the function $F \in C^{m-1,1}(\widehat{B})$ to a function in $C^{m-1,1}(\R^n)$ with norm bounded by $C \| F \|_{C^{m-1,1}(\widehat{B})} \leq C C'  C_0 M \leq C'' C_0 M$ -- by abuse of notation, we denote this extension by the same symbol $F$. Then (D) $\| F \| \leq C'' C_0 M$. Furthermore, (B) and (C) continue to hold for this extension. From (B),(C), and (D) we conclude that $P \in \Gamma_{E \cap \widehat{B}}(x_0, f, C'' C_0 M)$; here, $C''$ is a constant determined by $m$ and $n$. This finishes the proof of \eqref{local_fp1}, with $C_1 = C'' C_0$.

\end{proof}

\subsection{Proof of Lemma \ref{main_lem}}\label{proof_main_lem}

Let $\widehat{R} := C_*^m R_{\huge}$. From property (c) of Lemma \ref{mdl} (applied with $\delta = \diam(B)$), we have 
\begin{align}\label{eqn:uc0}
&\cB_{x,\diam(B)}/V \subset \widehat{R} \cdot (\sigma(x) \cap \cB_{x,\diam(B)})/V \\
\label{eqn:uc1}
&\sigma(x) \cap V \subset \widehat{R} \cdot \cB_{x,  \diam(B)} \qquad\qquad\qquad (x \in 8B, \; B \in \cW).
\end{align} 
To generate similar inclusions for $\sigma_\ell(x) \supset \sigma(x)$ we introduce the idea of ``keystone balls'' which are balls in the Whitney cover $\cW$ for which a local finiteness principle is valid on a dilate of the ball by a factor of $A \gg \max\{ C_*, R_{\huge} \}$. See  \eqref{defn:A} for the definition of $A$. Using this local finiteness principle, we can derive an upper inclusion on $\sigma_\ell(x) \cap V$ when $x$ belongs to a dilate of a keystone ball. This information is transferred to the non-keystone balls by exploiting the ``quasicontinuity'' of the sets $\sigma_\ell(x)$ (see Lemma \ref{gamma_trans_lem}) and by the fact that every ball in the cover is close to a keystone ball (see Lemma \ref{key_geom_lem} below).

\subsubsection{Keystone balls}

We first  introduce the notion of a \emph{keystone ball} of $\cW$.

\begin{definition}
\label{key_defn}
A ball $B^\# \in \cW$ is keystone if $\diam(B) \geq \frac{1}{2} \diam(B^\#)$ for every $B \in \cW$ with $B \cap A\cdot B^\# \neq \emptyset$. Let $\cW^\# \subset \cW$ denote the set of all keystone balls.
\end{definition}

\begin{lemma} \label{keystone_lem}
For each ball $B \in \cW$ there exists a keystone ball $B^\# \in \cW^\#$ with $B^\# \subset 3A B$, $\dist(B,B^\#) \leq 2A\diam(B)$, and $\diam(B^\#) \leq  \diam(B)$.
\end{lemma}
\begin{proof}
If $B$ is itself keystone, take $B^\# = B$ to establish the result. Otherwise, let $B_1 = B$. Since $B_1$ is not keystone there exists $B_2 \in \cW$ with $B_2 \cap A B_1 \neq \emptyset$ and $\diam(B_2) < \frac{1}{2} \diam(B_1)$. Similarly, if $B_2$ is not keystone there exists $B_3 \in \cW$ with $B_3 \cap A B_2 \neq \emptyset$ and $\diam(B_3) < \frac{1}{2} \diam(B_2)$. We continue to iterate this process. As $\cW$ is finite, the process must terminate after finitely many steps. Thus we produce a sequence of balls $B_1, B_2, \cdots, B_J \in \cW$ with $B_{j} \cap A B_{j-1} \neq \emptyset$ and $\diam(B_j) < \frac{1}{2} \diam(B_{j-1})$ for all $j$, and with $B_J$ keystone. As $B_{j} \cap A B_{j-1}\neq \emptyset$ we have $\dist(B_{j-1},B_j) \leq \frac{A}{2} \diam(B_{j-1})$. Now estimate
\[
\begin{aligned}
\dist(B_1,B_J) &\leq \sum_{j=2}^J  \dist(B_{j-1},B_j) +  \sum_{j=2}^{J-1} \diam(B_j)  \leq \left(A/2 + 1 \right) \sum_{j=1}^J \diam(B_j)  \\
& \leq  (A + 2 ) \diam(B_1) \leq 2A\diam(B_1).
\end{aligned}
\]
Since $\diam(B_J) \leq \diam(B_1)$, we have $B_J \subset (2A+6) B_1 \subset 3A B_1$. We set $B^\# = B_J$ to finish the proof.
\end{proof}

We now define a mapping $\kappa : \cW_0 \rightarrow \cW^\#$ satisfying a few key properties. By hypothesis of Lemma \ref{main_lem}, there exists a ball $\widehat{B} \in \cW_0$ with $\diam(\widehat{B}) \leq \epsilon_A \diam(B_0)$,  where $\epsilon_A := 1/3 A^2$. By Lemma \ref{keystone_lem}, there exists a keystone ball $\widehat{B}^\#$ with $\widehat{B}^\# \subset 3 A \widehat{B}$ and $\diam(\widehat{B}^\#) \leq \diam(\widehat{B})$. To define the mapping $\kappa$, we proceed as follows: For each $B \in \cW_0$,
\begin{itemize}
\item If $\diam(B)  > \epsilon_A \diam(B_0)$ ($B$ is \emph{medium-sized}), set $\kappa(B) := \widehat{B}^\#$.
\item If $\diam(B) \leq \epsilon_A \diam(B_0)$ ($B$ is \emph{small-sized}), Lemma \ref{keystone_lem} yields a keystone ball $B^\#$ with $B^\# \subset 3A B$; set $\kappa(B) := B^\#$.
\end{itemize}

\begin{lemma}\label{key_geom_lem}
The mapping $\kappa : \cW_0 \rightarrow \cW^\#$ satisfies the following properties:  For any $B \in \cW_0$, (a) $\dist(B,\kappa(B)) \leq C_4 \diam(B)$, (b) $\diam(\kappa(B)) \leq \diam(B)$, and (c) $A \cdot \kappa(B) \subset 2 B_0$. Here, $C_4$ is a constant determined by $m$ and $n$.
\end{lemma}
\begin{proof}
Suppose $B$ is medium-sized. Then $\kappa(B) = \widehat{B}^\#$. As $\diam(B) > \epsilon_A \diam(B_0)$ and $B \subset B_0$, we have $9 (\epsilon_A)^{-1} B \supset B_0 \supset \widehat{B}$; furthermore, $\widehat{B}^\# \subset 3A \widehat{B}$. Thus, $ \widehat{B}^\# \subset 27 (\epsilon_A)^{-1} A B$, which gives us (a) for $C_4 = 27 (\epsilon_A)^{-1} A = 81 A^3$. Also, $\diam(\widehat{B}^\#) \leq \diam(\widehat{B}) \leq \epsilon_A \diam(B_0)  < \diam(B)$, which establishes (b). By assumption $\widehat{B} \subset B_0$ and $\diam(\widehat{B}) \leq \epsilon_A \diam(B_0)$, which implies $3 A^2 \widehat{B} \subset (1 + 3 \epsilon_A A^2) B_0 = 2 B_0$. Thus, $A \widehat{B}^\# \subset 3 A^2 \widehat{B} \subset 2 B_0$, which gives (c).

Now suppose $B$ is small-sized. Then we defined $\kappa(B) = B^\#$, where $B^\#$ is related to $B$ as in Lemma \ref{keystone_lem}. In particular,  $\dist(B,B^\#) \leq  2A \diam(B)$ and $\diam(B^\#) \leq \diam(B)$, yielding (a) and (b). Furthermore, $B^\# \subset 3A B$. By assumption, $B \subset B_0$ and $\diam(B) \leq \epsilon_A \diam(B_0)$, which implies $3 A^2 B \subset (1 + 3 \epsilon_A A^2) B_0 = 2 B_0$. Thus, $A B^\# \subset 3 A^2 B \subset 2 B_0$, yielding (c).
\end{proof}

This completes the description of the geometric relationship between the balls of $\cW_0$ and keystone balls in $\cW$.  We will next need a lemma about the shape of $\sigma_\ell(z_{B^\#})$ for a keystone ball $B^\#$.
\begin{lemma}\label{lemma1}
Let $B^\# \in \cW$ be a keystone ball with $AB^\# \subset 2B_0$.  Let $\chi = \lceil \log (D \cdot (180A)^n  + 1)/\log(D+1) \rceil$, and let $\ell \in Z_{\geq 0}$ with $\ell \geq \ell_{\old} + \chi$. There exists a constant $C \geq 1$ determined by $m$ and $n$ such that the Local Finiteness Principle holds on $A B^\#$ with constants $CC_{\old}$ and $\ell$, namely, $\Gamma_{\ell}(x,f,M) \subset  \Gamma_{E \cap A B^\#}(x,f,CC_{\old} M) $ for all $x \in A B^\#$ and $M>0$. In particular, by taking $f=0$ and $M=1$, we have 
\[
\sigma_{\ell}(x) \subset CC_{\old} \sigma(x,E \cap A B^\#) \; \; \mbox{for any } x \in A B^\#.
\]
\end{lemma}

\begin{proof}

Let $\cW(B^\#)$ be the set of all balls in $\cW$ that intersect $A B^\#$. Since $\cW$ is a Whitney cover of $2B_0$ and $AB^\# \subset 2 B_0$, we have that $\cW(B^\#)$ is a Whitney cover of $A B^\#$. The Local Finiteness Principle holds on $\frac{6}{5}B$ for all $B \in \cW(B^\#)$, with constants $C_\old$ and $\ell_\old$ (see Lemma \ref{old_fp_lem}). Therefore, the Local Finiteness Principle holds on $A B^\#$ with the constants $C_1 = C \cdot C_{\old}$ and $\ell_1 = \ell_\old + \lceil \frac{\log (D \cdot N  + 1) }{\log(D+1)} \rceil$, where $N = \# \cW(B^\#)$; here, $C$ is a constant determined only by $m$ and $n$. See Lemma \ref{fip_lem}.

We prepare to estimate $N = \# \cW(B^\#)$ by a volume comparison bound.

For any $B \in \cW(B^\#)$, we have $\diam(B) \geq \frac{1}{2} \diam(B^\#)$ by definition of the keystone balls -- furthermore, we claim that $\diam(B) \leq 10A \diam(B^\#)$. We proceed by contradiction: Suppose  $\diam(B) > 10 A \diam(B^\#)$ for some $B \in \cW(B^\#)$. Then $B \cap AB^\# \neq \emptyset$ (by definition of $\cW(B^\#)$). Combining the previous two sentences gives $\frac{6}{5} B \cap B^\# \neq \emptyset$. Then $\diam(B) \leq 8 \diam(B^\#)$ by definition of a Whitney cover, which gives a contradiction. 

For any $B \in \cW(B^\#)$ we have $B \cap A B^\# \neq \emptyset$ and $\diam(B) \leq 10 A \diam(B^\#)$, and therefore $B \subset 30A B^\#$.

We estimate the volume of $\Omega := \bigcup_{B \in \cW(B^\#)}  \frac{1}{3} B$ in two ways. First, note that $\mbox{Vol}(\Omega) \leq \mbox{Vol}(30A B^\#) = (30A)^n \mbox{Vol}(B^\#)$. Next, using that the collection $\{\frac{1}{3} B\}_{B \in \cW}$ is pairwise disjoint, $N= \# \cW(B^\#)$, and $\diam(B) \geq \frac{1}{2} \diam(B^\#)$ for $B \in \cW(B^\#)$, we have
\[
\mbox{Vol}(\Omega) = \sum_{B \in \cW(B^\#)} 3^{-n} \mbox{Vol}(B) \geq  N 6^{-n} \mbox{Vol}(B^\#).
\]
We conclude that $N \leq (180A)^n$. Therefore, $\ell_1 = \ell_\old +\lceil \frac{\log (D \cdot N + 1) }{ \log (D+1)} \rceil   \leq \ell_\old + \chi \leq \ell$. Recall that the Local Finiteness Principle holds on $A B^\#$ with constants $C_1$ and $\ell_1$. Thus, the Local Finiteness Principle holds on $A B^\#$ with constants $C_1$ and $\ell$.
\end{proof}

\begin{lemma}\label{mainlem_1}
If $\ell \in \Z_{\geq 0}$ satisfies $\ell \geq \ell_{\old} + \chi$, and if $B^\# \in \cW$ is a keystone ball satisfying $AB^\# \subset 2B_0$, then
\begin{equation}\label{eqn:mainlem_1}
\sigma_\ell(z_{B^\#})  \cap V \subset C  C_{\old} \cB_{z_{B^\#}, \diam(B^\#)}.
\end{equation}
Here, the constant $\chi \geq 1$ is determined by $m$ and $n$ as in Lemma \ref{lemma1}, and $C \geq 1$ is determined by $m$ and $n$. \end{lemma}
\begin{proof}
By Lemma \ref{lemma1}, and Lemma \ref{lemma2},
\begin{equation}\label{eq:1}
\begin{aligned}
\sigma_{\ell}(z_{B^\#})  \cap C_0 C_{\old}\cB_{z_{B^\#}, A \diam(B^\#)} & \subset C_0 C_{\old} \cdot ( \sigma(z_{B^\#},E \cap AB^\#)\cap \cB_{z_{B^\#}, A \diam(B^\#)} ) \\
& \subset C_1 C_0 C_{\old} \cdot \sigma(z_{B^\#}) \quad (\ell \geq \ell_{\old} + \chi).
\end{aligned}
\end{equation}
Here, $C_0 = C_{\ref{lemma1}}$ and $C_1 = C_{\ref{lemma2}}$ are constants determined by $m$ and $n$.

The inclusion $\sigma(z_{B^\#}) \cap V \subset \widehat{R} \cB_{z_{B^\#},\diam(B^\#)}$ is a consequence of property (c) of Lemma \ref{mdl}; here, $\widehat{R} = C_*^m R_{\huge}$. Applying this inclusion and taking the intersection with $V$ on  each side of  \eqref{eq:1}, we obtain
\begin{equation}\label{eq:3}
\sigma_\ell(z_{B^\#}) \cap V \cap (C_0 C_{\old} \cB_{z_{B^\#}, A \diam(B^\#)}) \subset C_1 C_0 C_{\old} \widehat{R}  \cB_{z_{B^\#}, \diam(B^\#)}.
\end{equation}
Note that $ A \cB_{z_{B^\#}, \diam(B^\#)} \subset \cB_{z_{B^\#},A \diam(B^\#)}$ for $A \geq 1$. If $A > C_1 \widehat{R}$ then \eqref{eq:3} yields 
\[
\sigma_\ell(z_{B^\#}) \cap V \subset C_1 C_0 C_{\old} \widehat{R} \cdot \cB_{z_{B^\#}, \diam(B^\#)}.
\]
Recall that $A = 2 C_{\ref{lemma2}} \widehat{R}$; see \eqref{defn:A}. Therefore, $A > C_1 \widehat{R}$ and the above analysis applies. This completes the proof of \eqref{eqn:mainlem_1} with $C = C_1 C_0 \widehat{R}$, $C_1 = C_{\ref{lemma2}}$, $C_0 = C_{\ref{lemma1}}$.

\end{proof}

\subsubsection{Finishing the proof of  Lemma \ref{main_lem}}\label{fin_up_sec}

Fix $\chi$ as in Lemma \ref{mainlem_1}. Recall the  constants $A$ and  $\epsilon_A$, determined by $m$ and $n$, are defined in \eqref{defn:A}.

Fix $\widetilde{B} \in \cW_0$ and $\widetilde{x} \in 3\widetilde{B}$, and fix $\ell \geq \ell_{\old} + \chi$. 

Consider the keystone ball $B^\# = \kappa(\widetilde{B}) \in \cW$. As in Lemma \ref{key_geom_lem}, we have $\diam(B^\#) \leq \diam(\widetilde{B})$, $\dist(B^\#,\widetilde{B}) \leq C_4 \diam(\widetilde{B})$, and $A B^\# \subset 2 B_0$. By Lemma \ref{mainlem_1},
\begin{equation}
\label{eqn:501}
\sigma_\ell(z_{B^\#}) \cap V \subset C  C_{\old} \cB_{z_{B^\#}, \diam(B^\#)} \subset C  C_{\old} \cB_{z_{B^\#}, \diam(\widetilde{B})}.
\end{equation}

We now use condition (c) of Lemma \ref{mdl} for the ball $B^\# \in \cW$, the point $x=z_{B^\#}$, and the lengthscale $\delta = \diam(\widetilde{B}) \geq \diam(B^\#)$. This condition together with the inclusion $\sigma(z_{B^\#}) \subset \sigma_\ell(z_{B^\#})$ yields
\begin{equation}\label{eqn:502}
\cB_{z_{B^\#},\diam(\widetilde{B})}/V  \subset \widehat{R} \cdot (\sigma_\ell(z_{B^\#}) \cap \cB_{z_{B^\#},\diam(\widetilde{B})})/V, \quad \widehat{R} = R_{\huge} \cdot (C_*)^m.
\end{equation}

We prepare to shift the inclusions \eqref{eqn:501} and \eqref{eqn:502} from the basepoint $z_{B^\#}$ to the point $\widetilde{x} \in 3 \widetilde{B}$. As $z_{B^\#} \in \frac{6}{5} B^\#$ (see condition (e) of Lemma \ref{mdl}), we have 
\begin{equation}\label{eqn:stuff1}
|z_{B^\#} - \widetilde{x}| \leq \dist(B^\#, \widetilde{B}) + 3 \diam(\widetilde{B}) + \frac{1}{5} \diam(B^\#) \leq C_5 \diam(\widetilde{B}).
\end{equation}

By Lemma \ref{gamma_trans_lem}, \eqref{eq:change_delta}, and \eqref{eqn:stuff1}, we have
\begin{equation}\label{eqn:503}
\sigma_{\ell+1}(\widetilde{x}) + \cB_{z_{B^\#}, \diam(\widetilde{B})} \subset \sigma_{\ell}(z_{B^\#}) + \widetilde{C} \cdot \cB_{z_{B^\#}, \diam(\widetilde{B})},
\end{equation}
where $\widetilde{C}$ is a constant determined by $m$ and $n$.

We prepare to apply Lemma \ref{lem:stabv} to the convex subset $\Omega = \sigma_\ell(z_{B^\#})$ of the Hilbert space $X= (\cP,\langle \cdot, \cdot \rangle_{z_{B^\#},\diam(\widetilde{B})})$. We take $\lambda = \widetilde{C}$ in Lemma \ref{lem:stabv}. Note that the inclusions \eqref{eqn:501} and \eqref{eqn:502} imply the hypotheses (i), (ii) of Lemma \ref{lem:stabv} with $R = \widehat{R}$, $Z = C C_{\old}$. So we deduce that
\begin{equation}\label{eqn:504}
\left(\sigma_{\ell}(z_{B^\#}) + \widetilde{C} \cB_{z_{B^\#},\diam(\widetilde{B})} \right) \cap V \subset C C_{\old} \cdot(3 \widehat{R} \widetilde{C} + 1) \cB_{z_{B^\#},\diam(\widetilde{B})}.
\end{equation}
From \eqref{eqn:503} and \eqref{eqn:504}, we have
\begin{equation}\label{eqn:505}
(\sigma_{\ell+1}(\widetilde{x}) +  \cB_{z_{B^\#},\diam(\widetilde{B})}) \cap V \subset C C_{\old} \cdot \cB_{z_{B^\#},\diam(\widetilde{B})},
\end{equation}
where, as always, $C$ is a constant determined by $m$ and $n$.

Finally, note that  $C^{-1} \cdot \cB_{z_{\widetilde{B}},\diam(\widetilde{B})} \subset \cB_{z_{B^\#},\diam(\widetilde{B})} \subset C \cdot \cB_{z_{\widetilde{B}},\diam(\widetilde{B})}$; these inclusions follow from \eqref{trans_norm} and the estimate $|z_{\widetilde{B}} - z_{B^\#}| \leq C \diam(\widetilde{B})$. Therefore, \eqref{eqn:505} implies that
\[
(\sigma_{\ell+1}(\widetilde{x}) + \cB_{z_{\widetilde{B}},\diam(\widetilde{B})}) \cap V \subset C C_{\old} \cdot \cB_{z_{\widetilde{B}},\diam(\widetilde{B})},
\]
as desired. This finishes the proof of Lemma  \ref{main_lem}.

\section{The Main Induction Argument III: Putting it all together} \label{sec_glue}

Recall that our goal is to prove the inclusion \eqref{main-task}, stated as follows: For suitable constants $\ell^\# \in \Z_{\geq 0}$ and $C^\# \geq 1$, we have
\[
\Gamma_{\ell^\#}(x_0,f,1) \subset \Gamma_{E \cap B_0}(x_0,f, C^\#), \;\;\; \mbox{for all } x_0 \in B_0, \; f : E \rightarrow \R.
\]

We take $\ell^\#$ and $C^\#$ to satisfy \eqref{eqn:consts1} and \eqref{eqn:consts2} for constants $\overline{\chi}$ and $\widehat{C}$ to be defined momentarily. That is, $\ell^\# \geq \ell_{\old} + \overline{\chi}$ and $C^\# \geq \widehat{C} \cdot C_{\old}$. We choose $\overline{\chi}$ as in Lemma \ref{FMCJ_lem} so that this result is guaranteed to hold. We will choose $\widehat{C}$ later in the argument.

Once we prove the containment \eqref{main-task} as described above, we will have established the Local Finiteness Principle on $B_0$. This will complete the Main Induction Argument.  

Continuing with the argument outlined in the beginning of section \ref{sec_fmcj2}, we fix $P_0 \in \Gamma_{\ell^\#}(x_0,f,1)$. We apply Lemma \ref{mdl} to the data $(B_0, x_0, K, E, f, \ell^\#, P_0)$ to obtain a Whitney cover $\cW$ of $2B_0$, a DTI subspace $V \subset \cP$, and families $\{P_B\}_{B \in \cW}$ and $\{z_B\}_{B \in \cW}$. Recall that we defined in \eqref{defn:W0} the subcover $\cW_0 = \{ B \in \cW : B \cap B_0 \neq \emptyset \}$ of $\cW$; note that $\cW_0$ is a cover of $B_0$. 

Condition (f) in Lemma \ref{mdl} states that $P_B \in \Gamma_{\ell^\# - 1}(z_B,f,C)$ for all $B \in \cW$. By Lemma \ref{old_fp_lem} and by the fact that $\ell^\# -1 \geq \ell_\old$, we have $\Gamma_{\ell^\# - 1}(z_B,f,C) \subset \Gamma_{\ell_\old}(z_B,f,C) \subset \Gamma_{E \cap \frac{6}{5}B}(z_B,f,C \cdot C_{\old})$. So,
\[
P_B \in \Gamma_{E \cap \frac{6}{5}B}(z_B,f,C \cdot C_{\old}) \qquad (B \in \cW).
\]
By definition of the set $\Gamma_{E \cap \frac{6}{5}B}(\cdots)$, there exists $F_B \in C^{m-1,1}(\R^n)$ with
\begin{equation} \label{aaa1}
\left\{
\begin{aligned}
&F_B = f \mbox{ on } E \cap (6/5)B, \; J_{z_B} F_B = P_B, \mbox{ and }\\
&\| F_B \| \leq C \cdot C_{\old} \qquad (B \in \cW).
\end{aligned}
\right.
\end{equation} 
Since $\ell^\# \geq \ell_{\old} + \overline{\chi}$, we may apply Lemma \ref{FMCJ_lem} to conclude that
\begin{equation}
\label{aaa2}
\begin{aligned}
|J_{z_B} F_B - J_{z_{B'}} F_{B'}|_{z_B,\diam(B)} &= |P_B - P_{B'} |_{z_B, \diam(B)} \leq \overline{C} \cdot C_{\old} \\
& (B,B' \in \cW_0, \;  (6/5) \cdot  B \cap (6/5) \cdot B' \neq \emptyset).
\end{aligned}
\end{equation}
Let $\{\theta_B\}_{B \in \cW_0}$ be a partition of unity on $B_0$ adapted to the cover $\cW_0$ of $B_0$; see Lemma \ref{pou_lem1}. Define
\[
F = \sum_{B \in \cW_0} F_B \theta_B \mbox{ on }B_0.
\]
By Lemma \ref{pou_lem2} and by the conditions \eqref{aaa1} and \eqref{aaa2}, the function $F \in C^{m-1,1}(B_0)$  satisfies  $\| F \|_{C^{m-1,1}(B_0)} \leq C' \cdot C_{\old}$ and $F=f$ on $E \cap B_0$. 

If $x_0 \in \frac{6}{5}B$ for some $B \in \cW$ then  $z_B = x_0$ (by definition of the family $\{z_B\}$) and $P_B = P_0$ (by definition of the family $\{P_B\}$; see condition (e) in Lemma \ref{mdl}). Thus, $J_{x_0} F_B = P_0$ whenever $x_0 \in \frac{6}{5}B$. Therefore, 
\[
\begin{aligned}
J_{x_0} F &= \sum_{B \in \cW_0 : x_0 \in \frac{6}{5}B} J_{x_0} (F_B \theta_B) = \sum_{B \in \cW_0 : x_0 \in \frac{6}{5}B} J_{x_0} F_B \odot_{x_0} J_{x_0} \theta_B \\
& =  \sum_{B \in \cW_0 : x_0 \in \frac{6}{5}B} P_0 \odot_{x_0} J_{x_0} \theta_B = P_0 \odot_{x_0} 1 = P_0.
\end{aligned}
\]
We now extend the function $F$ to all of $\R^n$ by a classical extension technique (e.g., Stein's extension theorem). This gives a function $\widehat{F} \in C^{m-1,1}(\R^n)$ with $\| \widehat{F} \| \leq  C \| F \|_{C^{m-1,1}(B_0)} \leq \widehat{C} \cdot C_{\old}$ and $\widehat{F} = F$ on $B_0$; here, $\widehat{C}$ is a constant determined only by $m$ and $n$. In particular, $\widehat{F} = f$ on $E \cap B_0$ and $J_{x_0} \widehat{F} = P_0$ (since $x_0 \in B_0$). Thus, $P_0 \in \Gamma_{E \cap B_0}(x_0,f,\widehat{C} \cdot C_{\old})$. 

We take the constant $\widehat{C}$ in \eqref{eqn:consts2} as in the previous paragraph. Thus, as  $C^\# \geq \widehat{C} \cdot C_{\old}$, we have
\[
P_0 \in \Gamma_{E \cap B_0}(x_0,f,C^\#)
\] 
Since $P_0 \in \Gamma_{\ell^\#}(x_0,f,1)$ was arbitrary, this  finishes the proof of the containment \eqref{main-task} (if $\ell^\#$ and $C^\#$ satisfy \eqref{eqn:consts1} and \eqref{eqn:consts2}).

\subsection{The dependence of constants on complexity} \label{dep_sec}

In order to obtain the explicit dependence of the constants in Theorem \ref{main_thm2} on the complexity of $E$, we will need to show that the constants $\ell^\# = \ell^\#(K)$ and $C^\# = C^\#(K)$ in the Local Main Lemma for $K$ depend linearly and exponentially (resp.) on $K$; see Remark \ref{comment1}.

Our inductive proof of the Local Main Lemma for $K$ requires that we take $\ell^\#$ and $C^\#$ to be constants that satisfy \eqref{eqn:consts1} and \eqref{eqn:consts2}. That is, we only require that $C^\#(K) \geq \widehat{C} \cdot C^\#(K-1)$ and $\ell^\#(K) \geq \ell^\#(K-1) + \overline{\chi}$. By induction on $K$, we can choose $C^\# = \check{C}^K$ and $\ell^\# = \overline{\chi} \cdot K$, for a constant $\check{C}$ determined by $m$ and $n$. This completes the proof of the claim in Remark \ref{comment1}


\bibliography{mybib}{}
\bibliographystyle{amsplain}

\end{document}